\documentclass[12pt]{article}

\usepackage[latin1]{inputenc}

\usepackage[T1]{fontenc}
\usepackage{latexsym}
\usepackage{exscale}

\usepackage{amsmath,amsfonts,amssymb,amsthm}

\usepackage{graphicx,color}

\newtheorem{thm}{Theorem}[section]
\newtheorem{lem}[thm]{Lemma}
\newtheorem{cor}[thm]{Corollary}
\newtheorem{prop}[thm]{Proposition}
\newtheorem{rem}[thm]{Remark}
\newtheorem{rems}[thm]{Remarks}

\newtheorem{dfn}[thm]{Definition}

\newcommand{\aver}[1]{-\hskip-0.46cm\int_{#1}}
\newcommand{\avert}[1]{-\hskip-0.38cm\int_{#1}}
\newcommand{\ind}{1\hspace{-2.5 mm}{1}}

\DeclareMathOperator{\supp}{supp}
\newcommand{\diam}{{\rm diam}}
\newcommand{\dist}{{\rm dist}}
\newcommand{\grad}{{\mathbf \nabla}}
\newcommand{\Lap}{{\Delta}}
\newcommand{\Div}{{\rm div\; }}

\newcommand{\ra}{{\; \rightarrow\; }}

\newcommand{\R}{{\mathbb R}}
\newcommand{\Rn}{{\R^n}}
\newcommand{\C}{{\mathbb C}}

\newcommand{\Cinftyz}{{C^\infty_0}}
\newcommand{\Lip}{{\rm Lip}}
\newcommand{\Lipz}{{\Lip_0}}

\newcommand{\Tone}{{\mathcal{T}_1}}
\newcommand{\bPhi}{{\bf \Phi}}

\newcommand{\Wone}{{W_1^1}}
\newcommand{\dotWone}{{\dot{W}_1^1}}
\newcommand{\Mone}{{M_1^1}}
\newcommand{\dotMonep}{{\dot{M}_p^1}}
\newcommand{\dotMone}{{\dot{M}_1^1}}
\newcommand{\tilMone}{{\widetilde{M}_{1}^{1}}}
\newcommand{\Loneloc}{{L_{1,{\rm loc}}}}
\newcommand{\Lone}{{L_1}}
\newcommand{\Lp}{{L_p}}
\newcommand{\Hone}{{H_1}}
\newcommand{\Honeato}{{H_{1,{\rm ato}}}}
\newcommand{\Honemax}{{H_{1,{\rm max}}}}
\newcommand{\Honemol}{{H_{{\rm mol},1}}}
\newcommand{\HSonemax}{{\dot{HS}^1_{\rm max}}}
\newcommand{\dotHSato}{{\dot{HS}_{\rm ato}^{1}}}
\newcommand{\dotHStato}{{\dot{HS}_{t,{\rm ato}}^{1}}}
\newcommand{\HStato}{{HS_{t,\rm ato}^{1}}}
\newcommand{\LStato}{{LS^{1}_{t,{\rm ato}}}}

\newcommand{\Bibar}{{\overline{B_i}}}
\newcommand{\Bij}{{B_i^j}}
\newcommand{\Bijbar}{{\overline{\Bij}}}
\newcommand{\Bjk}{{B_k^j}}
\newcommand{\Bjkbar}{{\overline{\Bjk}}}
\newcommand{\Bjkprime}{{(\Bjk)'}}
\newcommand{\cij}{{c_i^j}}
\newcommand{\cjk}{{c_k^j}}
\newcommand{\elljk}{{\ell_k^j}}

\newcommand{\cM}{{\mathcal{M}}}
\newcommand{\Mq}{{\cM_q}}
\newcommand{\fplus}{{f^{+}}}
\newcommand{\fstar}{{f^{\star}}}

\newcommand{\qstar}{{q^*}}

\makeatletter
\renewcommand\@biblabel[1]{#1.}
\makeatother
\begin{document}
\allowdisplaybreaks

\title{An atomic decomposition of the Haj{\l}asz Sobolev space $\Mone$ on manifolds
 \footnote{Project funded in part by the Natural Sciences
and Engineering Research Council, Canada, the Centre de recherches
math\'ematiques and the Institut des sciences math\'ematiques,
Montreal.} }

\author{N. Badr\\
{\small Universite de Lyon; CNRS;} \\
{\small Universit\'e  Lyon 1, Institut Camille Jordan}\\
{\small 43 boulevard du 11 novembre 1918}\\
{\small F-69622 Villeurbanne cedex, France.}
\\{\small badr@math.univ-lyon1.fr} \and
G. Dafni\footnote{corresponding author}\\
{\small Department of Mathematics and Statistics}\\{\small Concordia
University}\\{\small 1455 de Maisonneuve Blvd. West}\\{\small
Montr\'eal, QC, Canada H3G1M8}\\
{\small gdafni@mathstat.concordia.ca}
\\{\small ph (514) 848-2424 x.3216, fax (514) 848-2831}
\\}
\date{}

\maketitle

\noindent

\begin{abstract} Several possible notions of Hardy-Sobolev
spaces on a Riemannian manifold with a doubling measure are
considered. Under the assumption of a Poincar\'e inequality, the
space $\Mone$, defined by Haj{\l}asz, is identified with a
Hardy-Sobolev space defined in terms of atoms. Decomposition results
are proved for both the homogeneous and the nonhomogeneous spaces.
\end{abstract}

{\small {\bf MS Classification} (2010): }{\small Primary: 42B30;
Secondary: 46E35, 58D15}\\

\section{Introduction}  The aim of this paper is to compare different
definitions of Hardy-Sobolev spaces on manifolds.  In particular, we
consider characterizations of these spaces in terms of maximal
functions, atomic decompositions, and gradients, some of which have
been shown in the Euclidean setting, and apply them to the $\Lone$
Sobolev space defined by Haj{\l}asz.

In the Euclidean setting, specifically on a domain $\Omega \subset
\Rn$, Miyachi \cite{miyachi} shows that for a locally integrable
function $f$ to have partial derivatives $\partial^\alpha f$ (taken
in the sense of distributions) belonging to the real Hardy space
$H_p(\Omega)$, is equivalent to a certain maximal function of $f$
being in $\Lp(\Omega)$.  Earlier work by Gatto, Jim\'enez and Segovia 
 \cite{gatto} on Hardy-Sobolev spaces, defined via powers of the Laplacian,
used a maximal function introduced by Calder\'on \cite{calderon} in
characterizing Sobolev spaces for $p > 1$ to extend his results to $p \leq 1$.
Calder\'on's maximal
function was subsequently studied by Devore and Sharpley
\cite{devsha}, who showed that it is pointwise equivalent to the
following variant of the sharp function. For simplicity we only give
the definition in the special case corresponding to one derivative
in $\Lone$, which is what this article is concerned with. We will
call this function the {\em Sobolev sharp maximal function} (it is
also called a ``fractional sharp maximal function'' in
\cite{kintuo}):
\begin{dfn} For $f\in \Loneloc$, define $Nf$ by
$$
Nf(x)=\sup_{B:\,x\in B}\frac{1}{r(B)}\aver{B}|f-f_B| d\mu,
$$
where $B$ denotes a ball, $r(B)$ its radius and $f_B$ the average of $f$ over $B$.
\end{dfn}
Another definition of Hardy-Sobolev spaces on $\Rn$,  using second
differences, is given by Strichartz \cite{strichartz}, who also
obtains an atomic decomposition.  Further characterizations of
Hardy-Sobolev spaces on $\Rn$ by means of atoms are given in
\cite{cho} and \cite{LY}.  For related work see \cite{janson}.

Several recent results provide a connection between Hardy-Sobolev
spaces and the $p=1$ case of Haj{\l}asz's definition of $\Lp$
Sobolev spaces on a metric measure space $(X,d,\mu)$:

\begin{dfn}[Haj{\l}asz] Let $1\leq p\leq\infty$. The (homogeneous) Sobolev space $\dotMonep$ is the set of all functions
$u\in \Loneloc$ such that there exists a measurable function $g\geq0$, $g\in \Lp$, satisfying
 \begin{equation}\label{Mp1}
 |u(x)-u(y)|\leq d(x,y)(g(x)+g(y))\;\mu-a.e.
\end{equation}
We equip $\dotMonep$ with the semi-norm
$$
 \Arrowvert u\Arrowvert_{\dotMonep}=\inf_{g\textrm{ satisfies} (\ref{Mp1})}\Arrowvert g\Arrowvert_{p}.
$$
\end{dfn}
In the Euclidean setting, Haj{\l}asz \cite{hajlasz1} showed the
equivalence of this definition with the usual one  for $1 < p \leq
\infty$.  For $p \in (n/n+1,1]$, Koskela and Saksman \cite{KS} proved
that $\dotMonep(\Rn)$ coincides with the homogeneous Hardy-Sobolev
space $\dot{H}^1_p(\Rn)$ defined by requiring all first-order
partial derivatives of $f$ to lie in the real Hardy space $H_p$ (the
same space defined by Miyachi \cite{miyachi}).  In recent work \cite{KYZ}, the Haj{\l}asz Sobolev spaces $\dot{M}^{s}_p$, 
for $0<s\leq 1$ and $\frac{n}{n+s}<p<\infty$, are characterized as homogeneous grand Triebel-Lizorkin spaces.

In the more general setting of a metric space with a doubling
measure, Kinnunen and Tuominen \cite{kintuo} show that Haj{\l}asz's
condition is equivalent to Miyachi's maximal function
characterization, extending to $p = 1$ a previous result of
Haj{\l}asz and Kinnunen \cite{hajkin} for $p > 1$:

\begin{thm}[\cite{hajkin},\cite{kintuo}]
\label{MN1} For $1\leq p<\infty$
$$
\dotMonep= \{f\in \Loneloc:\, Nf\in \Lp\}
$$
with
$$\|f\|_{\dotMonep}\sim\|Nf\|_{p}.
$$
Moreover, if $f\in \Loneloc$ and $Nf\in \Lone$, then $f$ satisfies
\begin{equation}\label{MN}
|f(x)-f(y)|\leq Cd(x,y) (Nf(x)+Nf(y))
\end{equation}
for $\mu-a.e. \,x,\, y$.
\end{thm}

We now restrict the discussion to a complete Riemannian manifold $M$
satisfying a doubling condition and a Poincar\'e inequality (see
below for definitions). In this setting, Badr and Bernicot
\cite{babe} defined a family of homogeneous atomic Hardy-Sobolev
spaces $\dotHStato$ and proved the following comparison
between these spaces:
\begin{thm}
\label{thm:comph} {\em (\cite{babe})} Let $M$ be a complete
Riemannian manifold satisfying a doubling condition and a Poincar\'e
inequality $(P_{q})$ for some $q>1$. Then $ \dot{HS}_{t,{\rm
ato}}^{1} \subset \dot{HS}_{\infty,{\rm ato}}^{1}$ for every $t\geq
q$ and therefore $\dot{HS}_{t_1,{\rm ato}}^{1}= \dot{HS}_{t_2,{\rm
ato}}^{1}$  for every $q\leq t_1,t_2\leq \infty$.
\end{thm}
In particular, under the assumption of the Poincar\'e inequality
$(P_1)$, for every $t > 1$ we can take $1 < q \leq t$ for which
$(P_{q})$ holds, so all the atomic Hardy-Sobolev spaces $\dotHStato$
coincide and can be denoted by $\dotHSato$.

The main result of this paper is to identify this atomic
Hardy-Sobolev space with Haj{\l}asz's Sobolev space for $p=1$:
\begin{thm}
\label{mainthm}
Let $M$ be a complete Riemannian manifold satisfying a doubling condition
and the Poincar\'e inequality $(P_1)$.
Then
$$\dotMone = \dotHSato.$$
\end{thm}

The definition of the atomic Hardy-Sobolev spaces, as well as the
doubling condition, the Poincar\'e inequality, and other
preliminaries, can be found in Section~\ref{prelims}. The proof of
Theorem~\ref{mainthm}, based on the characterization given by
Theorem~\ref{MN1} and a Calder\'on-Zygmund decomposition, follows in
Section~\ref{atomic-homo}. In Section~\ref{atomic-nonhomo}, a
nonhomogeneous version of Theorem~\ref{mainthm} is obtained.
Finally, in Section~\ref{sec:comp}, we characterize our
Hardy-Sobolev spaces in terms of derivatives. In particular, we show
that the space of differentials $df$ of our Hardy-Sobolev functions
coincides with the molecular Hardy space of differential one-forms
defined by Auscher, McIntosh and Russ \cite{AMR} (and by Lou and
McIntosh \cite{LM} in the Euclidean setting).

\section{Preliminaries}
\label{prelims} In all of this paper $M$ denotes a complete
non-compact Riemannian manifold. We write $T_xM$ for the tangent
space at the point $x \in M$, $\langle \cdot ,\cdot \rangle_x$ for
the Riemannian metric at $x$,  and $\mu$ for the Riemannian measure
(volume) on $M$.  The Riemannian metric induces a distance function
$\rho$ which makes $(M,\rho)$ into a metric space, and $B(x,r)$ will
denote the ball of radius $r$ centered at $x$ in this space.

Let $T_x^*M$ be the cotangent space at $x$, $\Lambda T_x^*M$ the
complex exterior algebra, and $d$ the exterior derivative acting on
$\Cinftyz(\Lambda T^*M)$.  We will work only with functions
($0$-forms) and hence for a smooth function $f$, $df$ will be a
$1$-form. In fact, in most of the paper we will deal instead with
the gradient $\grad f$, defined as the image of $df$ under the
isomorphism between $T_x^*M$ and $T_xM$ (see \cite{warner}, Section
4.10). Since this isomorphism preserves the inner product, we have
\begin{equation}
\label{gradient}
\langle df ,df \rangle_x = \langle \grad f,\grad f \rangle_x
\end{equation}

Letting $\Lp:=\Lp(M,\mu)$, $1 \leq p\leq \infty$, and denoting by
$|\cdot|$ the length induced by the Riemannian metric on the tangent
space (forgetting the subscript $x$ for simplicity), we can define
$\|\grad f\|_{p}:= \||\grad f|\|_{\Lp(M,\mu)}$  and, in view of
(\ref{gradient}), $\|df\|_{p} = \|\grad f\|_{p}$. If $d^*$ denotes
the adjoint of $d$ on $L_2(\Lambda T_x^*M)$, then the
Laplace-Beltrami operator $\Lap$ is defined by $dd^* + d^*d$.
However since $d^*$ is null on $0$-forms, this simplifies to $\Lap f
= d^*df$ on functions and we have, for $f,g \in \Cinftyz(M)$, using
(\ref{gradient}),
$$\langle\Delta f,g\rangle_{L_2(M)} = \int_M \langle\Delta f,g\rangle_x d\mu =
\int_M \langle df,dg\rangle_x d\mu = \langle \grad f ,\grad g
\rangle_{L_2(M)}.$$

We will use $\Lip(M)$ to denote the space of Lipschitz functions, i.e.\ functions $f$ satisfying, for some
$C < \infty$, the global Lipschitz condition
$$|f(x) - f(y)| \leq C\rho(x,y) \quad \forall \; x,y \in M.$$
The smallest such constant $C$ will be denoted by $\|f\|_\Lip$. By
$\Lipz(M)$ we will denote the space of compactly supported Lipschitz
functions. For such functions the gradient $\grad f$ can be defined
$\mu$-almost everywhere and is in $L_\infty(M)$, with $\|\grad
f\|_\infty \approx \|f\|_\Lip$ (see \cite{cheeger} for Rademacher's
theorem on metric measure spaces and also the discussion of upper
gradients in \cite{hajkosk}, Section 10.2).

\subsection{The doubling property}

\begin{dfn} Let $M$ be a Riemannian manifold. One says that $M$ satisfies the (global)
doubling property $(D)$ if there exists a
constant $C>0$, such that for all $x\in M,\, r>0 $ we have
\begin{equation*}\tag{$D$}
\mu(B(x,2r))\leq C \mu(B(x,r)).
\end{equation*}
\end{dfn}
\noindent Observe that if $M$ satisfies $(D)$ then
$$ \diam(M)<\infty\Leftrightarrow\,\mu(M)<\infty$$
(see \cite{ambrosio1}). Therefore if $M$ is a complete non-compact
Riemannian manifold satisfying $(D)$ then $\mu(M)=\infty$.
\begin{lem} Let $M$ be a Riemannian manifold satisfying $(D)$ and let
$s=\log_{2}C_{{\rm (D)}}$. Then for all $x,\,y\in M$ and $\theta\geq
1$
\begin{equation}\label{teta}
\mu(B(x,\theta R))\leq C\theta^{s}\mu(B(x,R)).
\end{equation}
\end{lem}

\begin{thm}[Maximal theorem, \cite{CW1}]\label{MIT}
Let $M$ be a Riemannian manifold satisfying $(D)$. Denote by $\cM$
the non-centered Hardy-Littlewood maximal function over open balls
of $M$, defined by
 $$\cM f(x):=\underset{\genfrac{}{}{0pt}{}{B \ \textrm{ball}}{x\in B}} {\sup} \ |f|_{B},$$
 where $\displaystyle f_{E}:=\aver{E}f d\mu:=\frac{1}{\mu(E)}\int_{E}f d\mu.$
Then for every  $1<p\leq \infty$, $\cM$ is $\Lp$ bounded and
moreover it is of weak type $(1,1)$. Consequently, for
$r\in(0,\infty)$, the operator $\cM_r$ defined by
$$ \cM_rf(x):=\left[\cM(|f|^r)(x) \right]^{1/r} $$
is of weak type $(r,r)$ and $\Lp$ bounded for all $r<p\leq \infty$.
\end{thm}

Recall that an operator $T$ is of weak type $(p,p)$ if there is $C>0$ such that for any $\alpha>0$, $\mu(\{x:\,|Tf(x)|>\alpha\})\leq \frac{C}{\alpha^p}\|f\|_p^p$.

\subsection{Poincar\'e inequality}
\begin{dfn}[Poincar\'{e} inequality on $M$] We say that a complete Riemannian manifold $M$ admits
\textbf{a Poincar\'{e} inequality $(P_{q})$} for some $q\in[1,\infty)$ if there exists a constant $C>0$ such that, for every function $f\in \Lipz(M)$ and every ball $B$ of $M$ of radius $r>0$, we have
\begin{equation*}\tag{$P_{q}$}
\left(\aver{B}|f-f_{B}|^{q} d\mu\right)^{1/q} \leq C r \left(\aver{B}|\grad f|^{q}d\mu\right)^{1/q}.
\end{equation*}
\end{dfn}

We also recall the following result
\begin{thm} {\em (\cite{hajlasz2}, Theorem 8.7)} Let $u \in \dotMone$ and
$g\in \Lone$ such that $(u,g)$ satisfies {\rm (\ref{Mp1})}. Take
$\frac{s}{s+1}\leq q <1$ and $\lambda>1$.  Then $(u,g)$ satisfies
the following Sobolev-Poincar\'e inequality: there is a constant
$C>0$ depending on $(D)$ and $\lambda$, independent of $(u,g)$ such
that for all balls $B$ of radius $r>0$,
\begin{equation}\label{Pg}
\left(\aver{B}|u-u_B|^{\qstar}d\mu\right)^{1/\qstar} \leq
Cr\left(\aver{\lambda B}g^q d\mu\right)^{1/q},
\end{equation}
where $\qstar=\frac{sq}{s-q}$.
\end{thm}
Applying this together with Theorem~\ref{MN1}, for $u\in \dotMone$
we have
\begin{equation}\label{PNq}
\left(\aver{B}|u-u_B|^{\qstar}d\mu\right)^{1/\qstar}\leq Cr\left(\aver{\lambda B}(Nu)^q d\mu\right)^{1/q}
\end{equation}
for all balls $B$.

\subsection{Comparison between $Nf$ and $|\grad f|$}
The following Proposition shows that the maximal function $Nf$ controls
the gradient of $f$ in the pointwise almost-everywhere sense.  In
the Euclidean setting this result was demonstrated by Calder\'on (see
\cite{calderon}, Theorem 4) for his maximal function $N(f,x)$
(denoted by $\fstar$ in Section~\ref{sec:Moneone} below), which was shown to
be pointwise equivalent to our $Nf$ by Devore and Sharpley (see also
the stronger inequality (5.5) in  \cite{miyachi}, which bounds the
maximal function of the partial derivatives).

Recall that if $u \in \Cinftyz(M)$, given any smooth vector field
$\bPhi$ with compact support, we can write, based on
(\ref{gradient}) and the definition of $d^*$,
$$
\int_M \langle \grad u,\bPhi\rangle_x d\mu := \int_M
\langle du ,\omega_\bPhi\rangle_x d\mu = \int_M
u(d^*\omega_\bPhi) d\mu,
$$
where $\omega_\bPhi$ is the $1$-form corresponding to ${\bf \Phi}$
under the isomorphism between the tangent space $T_xM$ and the
co-tangent space $T_x^*M$ (see \cite{warner}, Section 4.10).
Denoting $d^*\omega_\bPhi$ by $\Div \bPhi$, we can define, for $u\in
\Loneloc$, the gradient $\grad u$ in the sense of distributions by
\begin{equation}
\label{graddist}
\langle \grad u,\bPhi\rangle :=
-\int_M u(\Div \bPhi) d\mu
\end{equation}
for all smooth vector fields $\bPhi$ with compact support
(see \cite{MPPP}). When $M$
is orientable,  $\Div \bPhi$ is given by $*d* \omega_\bPhi$ with
$*$ the Hodge star operator (see \cite{warner}), and in the Euclidean case this
corresponds to the usual notion of divergence of a vector field.

\begin{prop} \label{nabN}Assume that M satisfies $(D)$, and suppose $u\in \Loneloc$
with $Nu\in \Lone$.  Then $\grad u$, initially defined by
(\ref{graddist}), is given by an $\Lone$ vector field and satisfies
$$
|\grad u| \leq CNu\qquad   \mu-{\rm a.e.}
$$
\end{prop}

\begin{proof}
Fix $r>0$. We begin with a covering of $M$  by balls $B_i=B(x_i,r)$,
$i=1,2...$ such that
\begin{itemize}
\item[1.]  $M\subset \cup_iB_i$,
\item[2.] $\sum_i \ind_{6B_i}\leq K$.
\end{itemize}
Note that the constant $K$ can be taken independent of $r$. Then we
take $\{\varphi_i\}_i$ a partition of unity related to the covering
$\{B_i\}_i$ such that $0\leq \varphi_i\leq 1$, $ \varphi_i=0$ on
$(6B_i)^c$, $ \varphi_i\geq c $ on $ 3B_i$ and $\sum_i \varphi_i=1$.
The $\varphi_i$'s are $C/r$ Lipschitz. For details concerning this
covering we refer to \cite{FHK}, \cite{kintuo}, \cite{hei},
\cite{CW2}. Now let (see \cite{FHK}, p.~1908 and \cite{kintuo},
Section 3.1)
\begin{equation}
\label{dfn:ur}
u_r(x)=\sum_j  \varphi_j(x)u_{3B_j}.
\end{equation}
The sum is locally finite and defines a Lipschitz function so we can take its
gradient and we have, for $\mu$-almost every $x$,
\begin{align}
\label{urmaxN}
|\grad u_r(x)|&=|\sum_j  \grad \varphi_j(x)u_{3B_j}| \nonumber
\\
&=|\sum_{\{j : x\in 6B_j\}}  \grad
\varphi_j(x)(u_{3B_j}-u_{B(x,9r)})| \nonumber
\\
&
\leq C K\frac{1}{r}  \aver{B(x,9r)}|u-u_{B(x,9r)}|d\mu \nonumber
\\
&\leq CK Nu(x).
\end{align}
We used the fact that $\sum \grad \phi_j = 0$ and that for $x \in
6B_j$, $3B_j\subset B(x,9r)$.

To see that $u_r\rightarrow u$ $\mu-a.e.$ and moreover in $\Lone$
when $r\rightarrow 0$  (see also \cite{FHK},p.\ 1908),  write, for
$x$ a Lebesgue point of $\mu$,
\begin{equation*}
|u_r(x)-u(x)|\leq  \sum_j |\varphi_j(x)||u(x)-u_{3B_j}|\leq
\sum_{\{j : x\in 6B_j\}}|u(x)-u_{3B_{j}}|\leq C K r
\Mq(Nu)(x)
\end{equation*}
where $\frac{s}{s+1}\leq q <1$. The last inequality follows from
estimates of $|u(x)-u_{B(x,9r)}|$ and $|u_{3B_j}-u_{B(x,9r)}|$, $x
\in 6B_j$, which are the same as estimates (12)-(14) in the proof of
Lemma 1 in \cite{kintuo}, using the doubling property and
(\ref{PNq}).

Now let $\bPhi$ be a smooth vector field with compact support. Using
the convergence in $\Lone$, the fact that $\Div \bPhi \in
\Cinftyz(M)$, and the estimate on $|\grad u_r|$ above, we have
 \begin{align*}
 |\int_M \langle \grad u,\bPhi\rangle_x d\mu|&=|\int_M u(\Div \bPhi) d\mu|
\\
& = |\lim_{r \ra 0}\int_M  u_r (\Div \bPhi)  d\mu|
\\
& \leq \limsup_{r \ra 0}\int_M |\grad u_r | |\bPhi| d\mu \leq
CK \int |Nu| |\bPhi| d\mu.
\end{align*}
Taking the supremum of the left-hand-side over all such $\bPhi$ with $|\bPhi| \leq 1$,
we get that the total variation of $u$ is bounded (see \cite{MPPP}, (1.4), p.\ 104), i.e.
$$|Du|(M) \leq C\|Nu\|_{\Lone(M)} < \infty,$$
hence $u$ is a function of bounded variation on $M$, and $|Du|$ defines a
finite measure on $M$.  We can write the distributional gradient as
$$\langle \grad u,\bPhi\rangle = \int_M \langle X_u,\bPhi\rangle_x d|Du|$$
for some vector field $X_u$ with $|X_u| = 1$ a.e.\ (see again
\cite{MPPP}, p.\ 104 where this is expressed in terms of the
corresponding $1$-form $\sigma_u$). Moreover, from the above
estimates and the fact that $Nu \in \Lone$, we further deduce that
the measure $|Du|$ is absolutely continuous with respect to the
Riemannian measure $\mu$, so there is an $\Lone$ function $g$ such
that we can write $\grad u = gX_u$, and $|\grad u| \leq C Nu$,
$\mu-a.e.$
\end{proof}

\begin{cor} \label{MW} Assume that $M$ satisfies $(D)$. Then $$
\dotMone \subset \dotWone.
$$
\end{cor}
 \begin{proof} The result follows from Proposition~\ref{nabN} and Theorem~\ref{MN1}.
 \end{proof}

\subsection{Hardy spaces}

We begin by introducing the maximal function characterization of the
real Hardy space $\Hone$.
\begin{dfn} \label{Hone-max}Let $f\in \Loneloc (M)$. We define its grand maximal function, denoted
by $\fplus$, as follows:
\begin{equation}
\label{def:fplus}
\fplus(x):=\sup_{\varphi \in \Tone(x)}\left|\int f\varphi d\mu\right|
\end{equation}
where $\Tone(x)$ is the set of all test functions $\psi \in \Lipz(M)$ such that for some ball
$B:=B(x,r)$ containing the support of $\psi$,
\begin{equation}
\label{Tone}
\|\psi\|_{\infty} \leq \frac{1}{\mu(B)}, \qquad  \|\grad \psi\|_{\infty}\leq \frac{1}{r\mu(B)}.
\end{equation}
Set $\Honemax(M) = \{f \in \Loneloc (M): \fplus \in \Lone(M)\}$.
\end{dfn}
While this definition assumes $f$ to be only locally integrable, 
by taking an appropriate sequence $\varphi_\epsilon \in \Tone(x)$,
the Lebesgue differentiation theorem implies that
\begin{equation}
\label{fleqfplus}
|f(x)| = \lim_{\epsilon \ra 0} \left|\int f\varphi_\epsilon\right| \leq \fplus(x) \; \mbox{ for $\mu$-a.e.}\; x,
\end{equation}
so $\Honemax(M) \subset \Lone(M)$.

Another characterization is given in terms of atoms (see \cite{CW2}).
\begin{dfn} \label{Hone-atoms}
 Fix $1<t\leq \infty$, $\frac{1}{t} + \frac{1}{t'} = 1$.
 We say that a function $a$ is an $\Hone$-atom if
\begin{itemize}
\item[$1$.] $a$ is supported in a ball $B$,
\item[$2$.] $\|a\|_{t} \leq \mu(B)^{-\frac{1}{t'}}$, and
\item[$3$.] $\int a d\mu=0$.
\end{itemize}
We say $f$ lies in the atomic Hardy space $\Honeato$ if $f$ can be
represented, in $\Lone(M)$, by
\begin{equation}
\label{Hone-atom-decomp}
f = \sum \lambda_j a_j
\end{equation}
for sequences of $\Hone$-atoms $\{a_j\}$ and scalars $\{\lambda_j\} \in \ell^1$.
Note that this representation is not unique and we define
$$\|f\|_{\Honeato} := \inf \sum |\lambda_j|,$$
where the infimum is taken over all atomic decompositions \rm{(\ref{Hone-atom-decomp}}).
\end{dfn}

A priori this definition depends on the choice of $t$.  However, we claim
\begin{prop}
\label{prop-Hone}
Let $M$ be a complete Riemannian manifold satisfying $(D)$. Then
 $$\Honeato(M) = \Honemax(M)$$
 with equivalent norms
 $$\|f\|_{\Honeato}\approx \|\fplus\|_1$$
(where the constants of proportionality depend on the choice of $t$).
\end{prop}
In the case of a space of homogeneous type $(X, d, \mu)$, this was
shown in \cite{maciassegovia} (Theorem 4.13) for a normal space of
order $\alpha$ and in \cite{uchiyama} (Theorem C) under the
assumption of the existence of a family of Lipschitz kernels (see
also the remarks following Theorem (4.5) in \cite{CW2}).  For the
manifold $M$ this will follow as a corollary of the atomic
decomposition for the Hardy-Sobolev space below.  We first prove the
inclusion
\begin{equation}
\label{Honeato-Honemax}
\Honeato(M) \subset \Honemax(M).
\end{equation}

\begin{proof} We show that  if $f\in \Honeato$ then $\fplus\in \Lone$.  Let $t>1$ and $a$
be an atom supported in a ball $B_0=B(x_0,r_0)$. We want to prove
that $a^+ \in \Lone$. First take $x\in 2B_0$. We have $a^+(x)=
\sup\limits_{\varphi \in \Tone(x)} \left|\int_{B}a \varphi
d\mu\right|\leq C \cM( a)(x)$. Then by the $L_t$-boundedness of the
Hardy-Littlewood maximal function for $t > 1$ (Theorem~\ref{MIT})
and the size condition on $a$,
\begin{align}
\int_{2B_0}|a^+(x)|d\mu & \leq \mu(B_0)^{1/t'}\left(\int_{2B_0}|a^+|^t d\mu\right)^{1/t}
\leq C\mu(B_0)^{1/t'}\|\cM a\|_t\nonumber \\
&\leq C_t\mu(B_0)^{1/t'}\|a\|_t
\leq C_t.
\label{doubleball}
\end{align}
Note that the constant depends on $t$ due to the dependence
of the constant in the boundedness of the Hardy-Littlewood maximal function, which blows up as $t \ra 1^+$.

Now if $x\in M\setminus 2B_0$, there exists $k\in \mathbb{N}^*$ such
that $x\in C_k(B_0):=2^{k+1}B_0\setminus 2^kB_0$. Let $\varphi \in
\Tone(x)$ and  take a ball $B=B(x,r)$ such that $\varphi$ is
supported in and satisfies (\ref{Tone}) with respect to $B$. Using
the moment condition for $a$ and the Lipschitz bound on $\varphi$,
we get
\begin{align*}
\left| \int_{B}a \varphi d\mu\right|&=\left|\int_{B\cap B_0}a(y)(\varphi(y)-\varphi(x_0))d\mu(y)\right|
\\
& \leq C \int_{B\cap B_0}|a(y)| \frac{d(y,x_0)}{r\mu(B)} d\mu(y)\\
& \leq C \frac{r_0}{r\mu(B)} \|a\|_1.
\end{align*}
Note that for the integral not to vanish we must have $B \cap B_0 \neq \emptyset$.  We claim that
this implies
\begin{equation}\label{BB1}
r>2^{k-1}r_0 \mbox{ and } 2^{k+1}B_0\subset 8B.
\end{equation}
To see this, let $y \in B\cap B_0$. Then $ r > d(y,x)\geq
d(x,x_0)-d(y,x_0) \geq 2^kr_0 - r_0 \geq 2^{k-1} r_0. $ Thus if
$d(z,x_0) < 2^{k+1}r_0$ then $d(z,x) \leq d(z,x_0) + d(x,x_0) <
2^{k+1}r_0 + 2^{k+1}r_0 < 8r$ and we deduce that $2^{k+1}B_0 \subset
8B$. We then have $$ \mu(2^{k+1}B_0)\leq C8^s\mu(B)$$ by
(\ref{teta}). Using this estimate and the fact that $\|a\|_{1}\leq
1$, we have
\begin{align*}
\int_{x\notin 2B_0}|a^+|(x)d\mu &= \sum_{k\geq 1}\int_{C_k(B_0)}|a^+|(x) d\mu
\\
&\leq C\|a\|_1 \sum_{k\geq 1} \frac{8^s 2^{1-k}}{\mu(2^{k+1}B_0)}\mu(C_k(B_0))
\\
&
\leq C8^s \sum_{k\geq 1} 2^{1-k}
\\
&\leq C.
\end{align*}
Thus $a^+ \in \Lone$  with $\|a^+\|_{1}\leq C_t$.

Now  for $f\in\Honeato$, take an atomic decomposition of $f$ as in
(\ref{Hone-atom-decomp}). By the convergence of the series in
$\Lone$, we have, for each $x$ and each $\varphi \in \Tone(x)$,
$$\left|\int f\varphi d\mu\right| \leq \sum |\lambda_j| \left|\int a_j \varphi d\mu\right| \leq \sum |\lambda_j| a_j^+(x)$$
so $\fplus$ is pointwise dominated by  $\sum |\lambda_j| a_j^+$, giving
$$\|\fplus\|_{1}\leq \sum_j|\lambda_j|\|a_j^+\|_1\leq C_t\sum_j|\lambda_j|.$$
Taking the infimum over all the atomic decompositions of $f$ yields
 $\|\fplus\|_{1}\leq C_t \|f\|_{\Hone}$.\end{proof}

The proof of the converse, namely that if $\fplus \in \Lone$  then
$f\in \Honeato$, relies on an atomic decomposition and will follow from
the proof of Proposition~\ref{AHS} below.

\subsection{Atomic Hardy-Sobolev spaces}
In \cite{babe}, the authors defined atomic Hardy-Sobolev spaces. Let
us recall their definition of homogeneous Hardy-Sobolev atoms.
These are similar to $\Hone$ atoms but instead of the usual $L_t$ size
condition they are bounded in the Sobolev space $\dot{W}^{1}_t$.
\begin{dfn}[\cite{babe}] \label{HHSA}
 For $1<t\leq \infty$, $\frac{1}{t} + \frac{1}{t'} = 1$,
 we say that a function $a$ is a homogeneous Hardy-Sobolev $(1,t)$-atom if
\begin{itemize}
\item[$1$.] $a$ is supported in a ball $B$,
\item[$2$.] $\|a\|_{\dot{W}^{1}_t}:=\|\grad a\|_{t}\leq \mu(B)^{-\frac{1}{t'}}$, and
\item[$3$.] $\int a d\mu=0$.
\end{itemize}
\end{dfn}
They then define, for every $1< t \leq \infty$, the homogeneous
Hardy-Sobolev space $\dotHStato$ as follows: $f\in
\dotHStato$ if there exists a sequence of
homogeneous Hardy-Sobolev $(1,t)$-atoms $\{a_{j}\}_{j}$ such that
\begin{equation}
\label{sum} f=\sum_{j}\lambda_{j}a_{j}
\end{equation}
with $\sum_{j}|\lambda_{j}|<\infty$. This space is equipped with the
semi-norm
$$
\|f\|_{\dotHStato}=\inf \sum_{j}|\lambda_{j}|,
$$
where the infimum is taken over all possible decompositions
(\ref{sum}).

\begin{rems}
\label{atom-rems}
\begin{enumerate}

\item  Since condition $2$ implies that the homogeneous Sobolev $\dotWone$
semi-norm of the atoms is bounded by a constant, the sum in
(\ref{sum}) converges in $\dotWone$ and therefore we can consider
$\dotHStato$ as its subspace.

\item Since we are working with homogeneous spaces, we
can modify functions by constants so the cancellation conditions are,
in a sense, irrelevant. As we will see below, and when comparing to
other definitions in the literature (see, for example, \cite{LY}),
condition $3$ can be replaced by one of the following:
\begin{itemize}
\item[$3^\prime$.] $\|a\|_1 \leq r(B)$, or
\item[$3^{\prime\prime}$.] $\|a\|_t \leq r(B) \mu(B)^{-\frac{1}{t'}}$,
\end{itemize}
where $r(B)$ is the radius of the ball $B$.  Clearly condition
$3^{\prime\prime}$ implies $3^\prime$, and conditions $2$ and $3$
imply $3^\prime$ (respectively $3^{\prime\prime}$) if we assume the
Poincar\'e inequality $(P_1)$ (respectively $(P_t)$).  It is most
common to consider the case $t=2$ under the assumption $(P_2)$.

\item As mentioned in the introduction, from Theorem~\ref{thm:comph}
we have that under $(P_1)$ all the spaces $\dotHStato$ can
be identified as one space $\dot{HS}^{1}_{ato}$.  As we will see, in
this case the atomic decomposition can be taken with condition $3^\prime$
instead of $3$.
\end{enumerate}
\end{rems}

\section{Atomic decomposition of $\dotMone$ and comparison with $\dotHStato$}
\label{atomic-homo} 
We begin by proving that under the Poincar\'e
inequality $(P_1)$,  $\dot{HS}_{ato}^1\subset \dotMone$.  While
under this assumption the space $\dot{HS}_{ato}^1$ is equivalent to
any one of the spaces $\dotHStato$ defined above, if we want to
consider the norms we need to fix some $t > 1$.

\begin{prop}
\label{comp1} Let $M$ be a complete Riemannian manifold satisfying
$(D)$ and $(P_1)$. Let  $1<t\leq \infty$ and $a$ be a homogeneous Hardy-Sobolev $(1,t)$-atom.
Then $a\in \dotMone$ with $\|a\|_{\dotMone}\leq C_t$, the constant
$C$ depending only on $t$, the doubling constant and the constant
appearing in $(P_1)$, and independent of $a$.

Consequently $\dotHStato\subset \dotMone$ with
$$
\|f\|_{\dotMone}\leq C_t\|f\|_{\dotHStato}.
$$
\end{prop}
\begin{proof}
 Let $a$ be an $(1,t)$-atom supported in a ball $B_0=B(x_0,r_0)$.
We want to  prove that $Na \in \Lone$. For $x\in 2B_0$ we have,
using $(P_1)$,
$$Na(x)= \sup\limits_{B:\;x\in B} \frac{1}{r(B)}
\aver{B}|a-a_B|d\mu\leq C \sup\limits_{B:\;x\in B}\aver{B}|\grad a
|d\mu =C \cM(|\grad a|)(x).$$ Then, exactly as in
(\ref{doubleball}), by the  $L_t$ boundedness of $\cM$ for $t>1$
(with a constant depending on $t$), and properties $1$ and $2$ of
$(1,t)$-Hardy-Sobolev atoms,
$$
\int_{2B_0}|Na(x)|d\mu \leq
C\mu(B_0)^{1/t'}\left(\int_{2B_0}(\cM(|\grad a
|))^td\mu\right)^{1/t} \leq C_t\mu(B_0)^{1/t'}\|\grad a\|_t \leq
C_t.
$$

Now if $x\notin 2B_0$, then there exists $k\in \mathbb{N}^*$ such
that $x\in C_k(B_0):=2^{k+1}B_0\setminus 2^kB_0$. Let $B=B(y,r(B))$
be a ball containing $x$. Then
\begin{align}
\frac{1}{r(B)} \aver{B}|a-a_B|d\mu&=\frac{1}{r(B)}
\frac{1}{\mu(B)}\left(\int_{B\cap B_0}|a-a_B|d\mu+\int_{B\cap
B_0^c}|a_B|d\mu\right) \nonumber
\\
& \leq  \frac{3}{r(B)} \frac{1}{\mu(B)}\int_{B\cap B_0}|a|d\mu.
\label{mcu}
\end{align}
From (\ref{BB1}) we have that $B\cap B_0 \neq \emptyset$ implies
$r(B)>2^{k-1}r_0$ and $\mu(2^{k+1}B_0)\leq C8^s\mu(B)$. This,
together with the doubling and Poincar\'e assumptions $(D)$ and
$(P_1)$, the cancellation condition $3$ for $a$ and the size
condition $2$ for $\grad a$, yield
\begin{equation*}
Na(x) \leq
\frac{3}{2^{k-1}r_0}\frac{8^s}{\mu(2^{k+1}B_0)}\int_{B_0}|a|d\mu
\leq \frac{3}{2^{k-1}}\frac{8^s}{\mu(2^{k+1}B_0)}\int_{B_0}|\grad
a|d\mu \leq 3\frac{ 2^{-k+1} 8^s}{\mu(2^{k+1}B_0)}.
\end{equation*}
Note that at this point we could have used condition $3^\prime$ (see
Remarks~\ref{atom-rems}) instead of conditions $2,3,(D)$ and
$(P_1)$.

Therefore
\begin{align*}
\int_{x\notin 2B_0}|Na|(x)d\mu &= \sum_{k\geq 1}\int_{C_k(B_0)}|Na|(x) d\mu
\leq C 8^s \sum_{k\geq 1}
2^{-k+1}\int_{C_k(B_0)}\frac{1}{\mu(2^{k+1}B_0)}d\mu(x)
\\
&
\leq C8^s \sum_{k\geq 1} 2^{-k+1} = C_s.
\end{align*}
Thus $Na \in \Lone$  with $\|Na\|_{1}\leq C_{s,t}$.

Now if $f\in{\dotHStato}$, take an atomic decomposition of $f$: $f=
\sum_j\lambda_ja_j$ with $a_j$ $(1,t)$-atoms and $\sum_j|\lambda_j|
< \infty$.  Then the sum $\sum_j \lambda_j Na_j$ converges
absolutely in $\Lone$ so by Theorem~\ref{MN1} the sequence of
functions $f_k = \sum_{j=1}^k\lambda_ja_j$ has a limit, $g$, in the
Banach space $\dotMone$.  By Proposition~\ref{nabN}, this implies
convergence in $\dotWone$.  Since (as pointed out in
Remarks~\ref{atom-rems}) the convergence of the decomposition $f=
\sum_j\lambda_ja_j$ also takes place in $\dotWone$, we get that $f =
g$ in $\dotWone$.  This allows us to consider $f$ as a (locally
integrable) element of $\dotMone$, take $Nf$ and estimate
$$\|Nf\|_{1}\leq
\sum_j|\lambda_j|\|Na_j\|_1\leq C_t\sum_j|\lambda_j|.$$
Taking the
infimum over all the atomic decompositions of $f$ yields
 $\|Nf\|_{1}\leq C_t\|f\|_{\dotHStato}$.
\end{proof}

\begin{rem}\label{remcomp1}
As pointed out in the proof, Proposition~\ref{comp1} remains valid
if we take, for the definition of a $(1,t)$-atom, instead of
condition $3$ of Definition~\ref{HHSA}, condition $3^\prime$ or
$3^{\prime\prime}$ of Remarks~\ref{atom-rems}.
\end{rem}

Now for the converse, that is, to prove that $\dotMone\subset \dotHStato$,  we establish  an atomic decomposition  for functions $f\in \dotMone$. To attain this goal, we need a Calder\'on-Zygmund decomposition for such functions. We refer to \cite{AC} for the original proof of the Calder\'on-Zygmund decomposition for Sobolev spaces on Riemannian manifolds.
\begin{prop}[Calder\'on-Zygmund decomposition] \label{CZ}
Let $M$ be a complete Riemannian manifold satisfying $(D)$.
Let $f\in \dotMone$, $\frac{s}{s+1}< q<1$ and $\alpha >0$.
Then one can find a collection of balls $\{B_{i}\}_{i}$,
functions $b_{i}\in \Wone$ and a Lipschitz function $g$
such that the following properties hold:
\begin{equation*}
f = g+\sum_{i}b_{i},
\end{equation*}
\begin{equation}
|\grad g(x)|\leq C\alpha\quad \mbox{for}\; \mu-a.e.\; x\in M, \label{eg}
\end{equation}
\begin{equation*}
\supp b_{i}\subset B_{i}, \,\|b_{i}\|_1\leq C\alpha \mu(B_i)r_i,\,
\|\grad b_{i}\|_{q}\leq C\alpha\mu(B_{i})^{\frac{1}{q}},\label{eb}
\end{equation*}
\begin{equation}
\sum_{i}\mu(B_{i})\leq \frac{C}{\alpha}\int N f d\mu,
\label{eB}
\end{equation}
and
\begin{equation}
\sum_{i}\chi_{B_{i}}\leq K \label{rb}.
\end{equation}
The constants $C$ and $K$ only depend on the constant in $(D)$.
\end{prop}

\begin{proof}
Let $f\in \dotMone$,  $\frac{s}{s+1}< q<1$ and $\alpha >0$. Consider the open set
$$
\Omega=\{x:\, \Mq(Nf)(x)>\alpha\}.
$$
If $\Omega=\emptyset$, then set
$$
 g=f\;,\quad b_{i}=0 \, \text{ for all } i
$$
so that (\ref{eg}) is satisfied according to the Lebesgue differentiation theorem. Otherwise
\begin{align}\label{mO}
\mu(\Omega)&\leq \frac{C}{\alpha } \int_{M} \Mq(Nf)d\mu \nonumber
\\
&\leq   \frac{C}{\alpha} \int_{M} \left(\cM(Nf)^q\right)^{1/q}d\mu
\nonumber
\\
&\leq  \frac{C}{\alpha } \int_{M} Nf d\mu<  \infty.
\end{align}
We used the fact the $\cM$ is $L_{1/q}$ bounded since $1/q>1$ and
Theorem~\ref{MN1}.
 In particular $\Omega\neq M$ as $\mu(M)=+\infty$. 
 
 Let $F$ be the complement of
 $\Omega$. Since $\Omega$ is an open set distinct from $M$, let
$\{\underline{B_{i}}\}_i$ be a Whitney decomposition of $\Omega$
(see \cite{CW2}). That is, the $\underline{B_{i}}$ are pairwise
disjoint, and there exist two constants $C_{2}>C_{1}>1$, depending
only on the metric, such that
\begin{itemize}
\item[1.] $\Omega=\cup_{i}B_{i}$ with $B_{i}=
C_{1}\underline{B_{i}}$, and the balls $B_{i}$ have the bounded overlap property;
\item[2.] $r_{i}=r(B_{i})=\frac{1}{2}d(x_{i},F)$ and $x_{i}$ is
the center of $B_{i}$;
\item[3.] each ball $\Bibar=C_{2}B_{i}$ intersects $F$ ($C_{2}=4C_{1}$ works).
\end{itemize}
For $x\in \Omega$, denote $I_{x}=\left\lbrace i:x\in B_{i}\right\rbrace$.
By the bounded overlap property of the balls $B_{i}$, we have that
$\sharp I_{x} \leq K$, and moreover, fixing $k\in I_{x}$,  $\frac{1}{r_i}\leq r_k\leq 3r_i$ and $B_{i}\subset 7B_{k}$ for all $i\in I_{x}$.

Condition (\ref{rb}) is nothing but the bounded overlap property of the $B_{i}$'s  and (\ref{eB}) follows from (\ref{rb}) and  (\ref{mO}).
Note also that using the doubling property, we have
\begin{equation}\label{faze}
\int_{B_{i}} |Nf|^{q}d\mu \leq C \mu(B_i) \aver{\Bibar} |Nf|^{q}d\mu \leq  \mu(B_i) \cM_q^q(Nf)(y) \leq C \alpha^ {q}\mu(B_{i})
\end{equation}
for some $y \in \Bibar\cap F$, whose existence is guaranteed by property $3$ of the Whitney
decomposition.

Let us now define the functions $b_{i}$. For this, we construct a
partition of unity $\{\chi_{i}\}_{i}$ of $\Omega$ subordinate to the
covering $\{B_{i}\}_i$. Each  $\chi_{i}$ is a Lipschitz function
supported in $B_{i}$ with $0 \leq \chi_i \leq 1$ and
$\displaystyle\|\grad \chi_{i} \|_{\infty}\leq \frac{C}{r_{i}}$ (see
for example \cite{FHK}, p.~1908).

We set $b_{i}=(f-c_i)\chi_{i}$ where
$c_i:=\frac{1}{\chi_{i}(B_{i})}\int_{B_{i}}f\chi_{i} d\mu$ and
$\chi_{i}(B_{i})$ means $\int_{B_{i}}\chi_{i}d\mu$, which is
comparable to $\mu(B_i)$. Note that by the properties of the
$\chi_i$ we have the trivial estimate
\begin{equation}
\label{bi-trivial} \|b_i\|_1 \leq \int_{B_{i}}|f-c_i|d\mu \leq
\int_{B_{i}} |f|d\mu +
\frac{\mu(B_i)}{\chi_{i}(B_{i})}\int_{B_{i}}|f| d\mu \leq C\int
\ind_{B_{i}} |f|d\mu,
\end{equation}
but we need a better estimate, as follows:
\begin{align}\label{cij1}
\|b_i\|_1
&\leq \frac{1}{\chi_{i}(B_{i})}\int_{B_{i}}\left|\int_{B_{i}}(f(x)-f(y))\chi_{i}(y)d\mu(y)\right|d\mu(x)\nonumber
\\
&\leq \frac{1}{\chi_{i}(B_{i})}\int_{B_{i}}\int_{B_{i}}|f(x)-f(y)|d\mu(y)d\mu(x)\nonumber
\\
&\leq 2\frac{\mu(B_i)}{\chi_{i}(B_{i})}\int_{B_{i}}|f(x)-f_{B_i}|d\mu(x)\nonumber
\\
&\leq C
r_i\left(\int_{\Bibar}|Nf|^qd\mu\right)^{1/q}\mu(B_i)\nonumber
\\
&\leq  C r_i\Mq(Nf)(y) \mu(B_i)\nonumber
\\
&\leq Cr_i \alpha  \mu(B_i),
\end{align}
as in (\ref{faze}). Here we have used the Sobolev-Poincar\'e
inequality (\ref{PNq}) with $\lambda = 4$ and the fact that $\qstar
> 1$.

Together with the estimate on $\|b_i\|_1$, we use the fact that $|\grad f|$ is
in $\Lone$ (see  Proposition~\ref{nabN}) to bound  $\|\grad b_i\|_1$ and conclude that $b_i \in \Wone$:
\begin{align}
\|\grad b_i\|_1&
\leq  \int_{B_i}|f-c_i| |\grad \chi_i|d\mu
+\int_{B_i}|\grad f|d\mu \nonumber
\\
&
\leq C \frac{1}{r_i}r_i \mu(B_i)
\left(\aver{4B_i}|Nf|^qd\mu\right)^{1/q}+\int_{B_i}|\grad f|d\mu
\nonumber
\\
&\leq C\alpha \mu(B_i)+ \int_{B_i}|\grad f|d\mu
< \infty.
\label{gradbi1}
\end{align}

Similarly, we can estimate $b_i$ in the Sobolev space
$\dot{W}^{1}_q$; note again  that  by Proposition~\ref{nabN}, $|\grad f|$ is
in $\Lone$ and can be bounded pointwise $\mu$-a.e.\ by $Nf$:
\begin{align}
\|\grad b_i\|_q &\leq  \|\,|(f-c_{i})\grad \chi_i|\, \|_q+\|\,|\grad f| \chi_i \|_q
\nonumber \\
&\leq \frac{\mu(B_i)^{\frac{1}{q}-1}}{\chi_{i}(B_{i})}\int_{B_{i}}\int_{B_{i}}|f(x)-f(y)| \chi_{i}(y)|\grad\chi_{i}(x)|d\mu(y)d\mu(x)+\left(\int_{B_i}|\grad f|^qd\mu\right)^{1/q}
\nonumber \\
&\leq
C\left(\aver{\Bibar}|Nf|^qd\mu\right)^{1/q}\mu(B_i)^{1/q}+\left(\int_{\Bibar}|N
f|^qd\mu\right)^{1/q}
\nonumber \\
&\leq C \alpha \mu(B_i)^{1/q}
\label{gradbiq}
\end{align}
by (\ref{faze}).

Set now $g=f-\sum_ib_i$. Since the sum is locally finite on $\Omega$,
$g$ is defined  almost everywhere on $M$ and $g=f$ on $F$. Observe that
$g$ is a locally integrable function on $M$. Indeed, let $\varphi\in L_{\infty}$
with compact support. Since $d(x,F)\geq r_{i}$ for $x \in \supp \,b_{i}$, we obtain
\begin{equation*} \int\sum_{i}|b_{i}|\,|\varphi|\,d\mu \leq
\Bigl(\int\sum_{i}\frac{|b_{i}|}{r_{i}}\,d\mu\Bigr)\,\sup_{x\in
M}\Bigl(d(x,F)|\varphi(x)|\Bigr)\quad
\end{equation*}
Hence by (\ref{cij1}) and the bounded overlap property,
$$\int\sum_{i}|b_{i}||\varphi|d\mu
\leq C\alpha\sum_{i}\mu(B_i) \sup_{x\in
M}\Bigl(d(x,F)|\varphi(x)|\Bigr) \leq CK\alpha\mu(\Omega) \sup_{x\in
M} \Bigl(d(x,F)|\varphi(x)|\Bigr). $$
Since $f\in L_ {1,loc}$, we
conclude that $g\in \Loneloc$.

It remains to prove (\ref{eg}).
Indeed, using the fact that on $\Omega$ we have $\sum \chi_i = 1$ and $\sum \grad \chi = 0$, we get
\begin{align}
\grad g &= \grad f -\sum_{i}\grad b_{i}\nonumber
\\
&=\grad f-(\sum_{i}\chi_{i})\grad f -\sum_{i}(f-c_{i})\grad
\chi_{i}\nonumber
\\
&=\ind_{F}\grad f -\sum_{i}(f-c_{i})\grad
\chi_{i}.
\label{gradg}
\end{align}

From Proposition~\ref{nabN}, the definition of $F$ and the Lebesgue differentiation theorem, we have that $\ind_{F} |\grad f|\leq \ind_{F} Nf \leq \alpha$, $\mu -$a.e. We claim that a similar estimate holds for
$$h=\sum_{i}(f-c_{i})\grad
\chi_{i},$$
i.e.\ $|h(x)|\leq C\alpha$ for all $x\in M$. For this, note first that by the properties of the balls $B_i$ and
the partition of unity, $h$ vanishes on $F$ and the sum defining $h$ is locally finite on $\Omega$.
Then fix  $x\in \Omega$ and let $B_{k}$  be some Whitney ball containing $x$.  Again using the fact that $\displaystyle \sum_{i}\grad\chi_{i}(x)=0$, we can replace $f(x)$ by any constant in the sum above, so we can write
$$h(x)=\sum_{i\in I_x} \left(\aver{7B_{k}}f d\mu -c_{i} \right)\grad
\chi_{i}(x).$$
For all $i,k\in I_{x}$, by the construction of  the Whitney collection, the balls $B_i$ and $B_k$ have equivalent radii and  $B_i \subset 7B_k$.
Thus
\begin{align}
\left|c_{i}-\aver{7B_{k}}f d\mu \right| & \leq \frac{1}{\chi_{i}(B_{i})} \int_{B_{i}} \left|f -\aver{7B_{k}}fd\mu\right| \chi_i  d\mu  \nonumber \\
 & \lesssim  \aver{7B_{k}} |f-f_{7B_{k}} |  d\mu \nonumber \\
 & \lesssim r_k\left(\aver{7\lambda B_k} |N f|^q d\mu \right)^{1/q} \nonumber \\
 & \lesssim \alpha r_k. \label{relim}
 \end{align}
We used $(D)$, (\ref{PNq}) , $\chi_i(B_i)\simeq \mu(B_i)$ and (\ref{faze})  for $7B_k$. Hence
\begin{align}
\label{h}
|h(x)| \lesssim \sum_{i\in I_{x}} \alpha r_k( r_{i})^{-1} \leq CK\alpha .
\end{align}
\end{proof}

\begin{prop} \label{AHS} Let $M$ be a complete Riemannian manifold satisfying $(D)$.
Let $f\in \dotMone$. Then for all $\frac{s}{s+1}< q<1$,
$\qstar=\frac{sq}{s-q}$, there is a sequence of  homogeneous
$(1,\qstar)$ Hardy-Sobolev  atoms $\{a_j\}_j$, and a sequence of scalars
$\{\lambda_j\}_j$, such that 
$$f=\sum_j \lambda_j a_j \quad \mbox{ in } \dotWone,
\mbox{ and } \quad \sum |\lambda_j| \leq C_q \|f\|_{\dotMone}.$$
Consequently, $\dotMone\subset \dot{HS}^1_{\qstar, ato}$ with $\|f\|_{\dot{HS}_{\qstar,{\rm ato}}^1}\leq C_q \|f\|_{\dotMone}$.
\end{prop}
\begin{rem} Note that for the inclusion $\dotMone\subset \dot{HS}^1_{\qstar, ato}$, we do not need to assume any additional hypothesis, such as a Poincar\'e inequality, on the doubling manifold.
\end{rem}
\begin{proof}[Proof of Proposition \ref{AHS}]
Let $f\in \dotMone$. We follow the general scheme of the atomic
decomposition for Hardy spaces, found in \cite{Stein}, Section
III.2.3.  For every $j\in \mathbb{Z}^*$, we take the
Calder\'on-Zygmund decomposition, Proposition~\ref{CZ}, for $f$ with
$\alpha= 2^j$. Then
$$ f=g^j+ \sum_i b_i^j$$
with $b_i^j$, $g^j$ satisfying the properties of Proposition \ref{CZ}.

We want to write
\begin{equation}
\label{convgj} f = \sum_{-\infty}^\infty (g^{j+1} - g^j)
\end{equation}
in $\dot{W}_1^1$.
First let us see that  $g^j \rightarrow f$ in as $j \ra \infty$.
Indeed, since the sum is locally finite we can write
$$\|\grad (g^j-f)\|_1=\|\grad( \sum_i b_i^j)\|_1
\leq \sum_i\|\grad b_i^j\|_1.$$
By (\ref{gradbi1}),
\begin{align}
\sum_i\|\grad b_i^j\|_1&
\leq CK2^j\mu(\Omega_j)+ K\int_{\Omega_j}|\grad f|d\mu \nonumber
\\
&=I_j+II_j.
\label{gradbij}
\end{align}
When $j\rightarrow \infty$, $I_j \rightarrow 0$ since
$\sum_j2^j\mu(\Omega_j)\approx\int\Mq(Nf)d\mu<\infty$. This also
implies $\Mq(Nf)$ is finite $\mu$-a.e., hence  $\bigcap \Omega_j =
\emptyset$ so $II_j \rightarrow 0$, since $|\grad f| \in \Lone$.

When $j \ra -\infty$, we want to show $\|\grad g_j\|_1 \ra 0$.  Breaking $\grad g$
up as in (\ref{gradg}), we know that
\begin{equation}
\label{gradgj}
\int_{F^j} |\grad g^j| = \int \ind_{F^j} |\grad f|  \leq \int_{\{Nf \leq 2^j\}} Nf  \ra  0,
\end{equation}
since $Nf \in \Lone$.  For the other part we have, by (\ref{h}),
\begin{equation}
\label{gradgj1}
\int_{\Omega^j} |\grad g^j| = \int |h(x)| \leq CK2^j \mu(\Omega^j) \ra 0
\end{equation}
from the convergence of $\sum 2^j \mu(\Omega^j)$, as above.

Denoting $g^{j+1}-g^j$ by $\ell^j$, we have $\supp \ell^j\subset \Omega_j$
so using the partition
of unity $\{\chi^j_k\}$ corresponding to the Whitney decomposition for $\Omega_j$,
we can write $f=\sum_{j,\,k}\ell^j\chi_k^j$ in $\dotWone$. Let us compute $\|\ell^j\chi_k^j\|_{\dot{W}_{\qstar}^1}$.
We have
$$
\grad (\ell^j\chi_k^j)=(\grad \ell^{j})\chi_k^j+\ell^j\grad \chi_k^j.
$$
From the estimate  $\|\grad g^j\|_{\infty}\leq C2^j$  it follows that
$\left(\avert{\Bjk}|\grad \ell^j|^{\qstar}d\mu\right)^{1/\qstar}\leq C2^j$, while
\begin{equation}\label{ljgradchi}
\ell^j\grad \chi_k^j= \left(\sum_{i: \Bjk\cap \Bij\neq \emptyset}
(f-\cij)\chi_{i}^j-\sum_{l: \Bjk\cap B_l^{j+1} \neq \emptyset}
(f-c_l^{j+1})\chi_l^{j+1}\right)\grad \chi_k^j.
\end{equation}
Observe  that since $ \Omega_{j+1} \subset \Omega_j$, for a fixed $k$, the
balls $B_l^{j+1}$ with
$\Bjk\cap B_l^{j+1}\neq \emptyset $ must have radii $r_l^{j+1} \leq c r_k^j$
for some constant $c$.
Therefore
$B_{l}^{j+1} \subset  \Bjkprime:= (1+2c)\Bjk$.  Moreover, by the properties of the
Whitney balls, given $\lambda > 1$ we can take $c$ sufficiently
large so that $\Bjkprime$ contains $\lambda \Bij$ for all $\Bij$ intersecting $\Bjk$.
Using this fact as well as (\ref{PNq}) and (\ref{faze}), and proceeding in the same
way as in the derivations of (\ref{cij1}) and (\ref{relim}), we get
\begin{align}
 \label{gradelljk}
(r_k^j)^{\qstar}\int_{\Bjk}| \ell^j\grad \chi_k^j|^{\qstar}d\mu
& \leq K^{\qstar-1} \int_{\Bjk}\Big(\sum_{i} \ind_{B_{i}^j}|f-\cij|^{\qstar}
+\sum_{l}
\ind_{B_{l}^{j+1}}|f-c_{l}^{j+1}|^{\qstar}\Big) d\mu \nonumber\\
&\leq K^{\qstar-1}\!\!\!\!\!\!\sum_{i: \Bjk\cap \Bij\neq \emptyset} \int_{B_{i}^j}|f-\cij|^{\qstar}d\mu \nonumber \\
&+K^{\qstar-1} \int_{\Bjkprime}\sum_{l}
\ind_{B_{l}^{j+1}}|f- f_{\Bjkprime} + f_{\Bjkprime} - c_{l}^{j+1}|^{\qstar} d\mu \\
&  \lesssim
K^{\qstar-1}\!\! \sum_{i: \Bjk\cap \Bij\neq \emptyset} (r_i^j 2^j)^{\qstar}\mu(\Bij)
+K^{\qstar} (r_k^j 2^j)^{\qstar}\mu(\Bjkprime) \nonumber \\
& \lesssim K^{\qstar} (r_k^j 2^j)^{\qstar}\mu(\Bjkprime). \nonumber
\end{align}
Therefore
\begin{equation}\label{ljgradchinorm}
 \left(\aver{\Bjkprime}|\ell^j \grad \chi_k^j|^{\qstar}d\mu\right)^{\frac{1}{\qstar}}\leq CK 2^j.
\end{equation}

The $\ell^j\chi_k^j$'s seem to be a good choice for our atoms but
unfortunately they do not satisfy the cancellation condition. If we
wanted to get atoms with property $3^\prime$ (see
Remarks~\ref{atom-rems}) instead of the vanishing moment condition
3, we could use (\ref{cij1}) to bound the $\Lone$ norm of
$\ell^j\chi_k^j$, then normalize as below. However, if we want to
obtain the vanishing moment condition,  we need to consider instead
the following decomposition of the $\ell^j$'s: $\ell^j=\sum_k
\elljk$ with
\begin{equation}
\label{def:elljk} \ell_k^j=(f-\cjk)\chi_k^j-\sum_l(f-c_l^{j+1})
\chi_l^{j+1}\chi_k^j+\sum_lc_{k,l}\chi_l^{j+1},
\end{equation}
where
$$
c_{k,l}:=
\frac{1}{\chi_l^{j+1}(B_l^{j+1})}\int_{B_l^{j+1}}(f-c_j^{l+1})\chi_l^{j+1}\chi_k^jd\mu.$$
First, this decomposition holds since $\sum_k\chi_k^j=1$ on the
support of $\chi_l^{j+1}$ and $\sum_k c_{k,l}=0$. Furthermore, the
cancellation condition
$$\int_{M}\ell_k^jd\mu=0$$
follows from the fact that
 $\int_M(f-\cjk)\chi_k^j d\mu=0$ and the definition of $c_{k,l}$, which immediately gives
$\int\left((f-c_l^{j+1})\chi_l^{j+1}\chi_k^j-c_{k,l}\chi_l^{j+1}\right)d\mu=0$.

Noting that $\ell_k^j$ is supported in the ball $\Bjkprime$ (see
above), let us estimate $ \|\grad
\ell_k^j\|_{L_{\qstar}(\Bjkprime)}$. Write
\begin{align*}
\grad \ell_k^j&= (\grad f) \chi_k^j+ (f-\cjk) \grad \chi_k^j
- \sum_l (f-c_l^{j+1})\grad \chi_l^{j+1}\chi_k^j\\
& - \sum_l(f-c_l^{j+1})\chi_l^{j+1}\grad \chi_k^j
-(\grad f) \ind_{\Omega_{j+1}}\chi_k^j
+ \sum_lc_{k,l}\grad \chi_l^{j+1}\\
& = \grad f(1 - \ind_{\Omega_{j+1}})\chi_k^j + ((f-\cjk) - \sum_l(f-c_l^{j+1})\chi_l^{j+1})\grad \chi_k^j\\
& - \sum_l (f-c_l^{j+1})\grad \chi_l^{j+1}\chi_k^j + \sum_lc_{k,l}\grad \chi_l^{j+1}.
\end{align*}
Since the first term, concerning the gradient of $f$,  is supported in $\Bjk\cap F_{j+1}$, we can use Proposition~\ref{nabN}, the definition of $F_{j+1}$ and the Lebesgue differentiation theorem to bound it, namely
$$\int_{\Bjk}|\grad f|^{\qstar}d\mu \leq  2^{(j+1)\qstar}\mu(\Bjk).$$
Recalling (\ref{ljgradchi}), we see that the estimate of the $L_{\qstar}$ norm of the second term is given
by (\ref{ljgradchinorm}).  The third term can be handled by the pointwise estimate (\ref{h}):
$$
\|\sum_l (f-c_l^{j+1})\grad \chi_l^{j+1}\chi_k^j\|_{\qstar} \leq CK 2^{j+1} \mu(\Bjk)^{1/\qstar}.$$
For $\sum_lc_{k,l}\grad \chi_l^{j+1}$, note first that $c_{k,l}=0$ when  $\Bjk\cap B_l^{j+1} = \emptyset$ and $|c_{k,l}|\leq C2^jr_l^{j+1}$ thanks to (\ref{cij1}). By the properties of the partition of unity, this
gives $|c_{k,l}\grad \chi_l^{j+1}| \leq C2^{j}$ for every $l$, and as the sum has at most $K$ terms at
each point we get the pointwise bound
$$|\sum_lc_{k,l}\grad \chi_l^{j+1}| \leq C K  2^{j},$$
from which it follows that
$$\|\sum_lc_{k,l}\grad \chi_l^{j+1}\|_{\qstar} \leq C K 2^j \mu(\Bjkprime)^{1/\qstar}.$$
Thus
\begin{equation} \label{ljk}
\|\grad \ell_k^j\|_{\qstar}\leq \gamma  2^j \mu(\Bjkprime)^{1/\qstar}.
\end{equation}

We now set $a_k^j=\gamma^{-1}2^{-j} \mu(\Bjkprime)^{-1}\ell_k^j$ and
$\lambda_{j,k}=\gamma 2^j \mu(\Bjkprime)$. Then  $
f=\sum_{j,k}\lambda_{j,k}a_k^j$, with $a_k^j$ being $(1,\qstar)$  homogeneous
Hardy-Sobolev atoms and
\begin{align*}
\sum_{j,k}|\lambda_{j,k}|&= \gamma \sum_{j,k}2^j\mu(\Bjkprime)
\\
&\leq \gamma' \sum_{j,k}2^j\mu(\underline{\Bjk})
\\
&\leq \gamma' \sum_{j}2^j\mu(\{x:\, \Mq(Nf)(x)>2^j\})
\\
&\leq C \int \Mq(Nf)d\mu
\\
&
\leq C_q \|Nf\|_1\sim \|f\|_{\dotMone}.
\end{align*}
We used that $\mu(\Bjkprime)\sim \mu(\underline{\Bjk})$ thanks to
$(D)$, and the fact that the $\underline{\Bjk}$ are disjoint.
\end{proof}

\begin{rem}
\label{rem-otheratoms}
 As pointed out in the proof following
(\ref{ljgradchinorm}), we can get an atomic decomposition as in
Proposition \ref{AHS}, but replacing the vanishing moment condition
$3$ of the atoms from Definition~\ref{HHSA} by condition $3^\prime$
in Remarks~\ref{atom-rems}. This does not assume a Poincar\'e
inequality.
\end{rem}

\textbf{\large{Conclusion:}} Let $M$ be a complete Riemannian manifold satisfying $(D)$.
Then
\begin{itemize}
\item[1.] for all $\frac{s}{s+1}<q<1$,
$$\dotMone\subset \dot{HS}_{\qstar,{\rm ato}}^1.$$

\item[2.] (Theorem~\ref{mainthm}) If moreover  we assume $(P_1)$, then
$$
\dotMone=\dotHStato
$$
for all $t>1$.
\end{itemize}

\section{The nonhomogeneous case}
\label{atomic-nonhomo} We begin by recalling the definitions of the
nonhomogeneous versions of the spaces considered above.
\begin{dfn} {\em (\cite{hajlasz2})}
Let $1\leq p\leq\infty$. The Sobolev space $M_{p}^{1}$ is the set of all functions $u\in L_{p}$ such that there exists a measurable function $g\geq0$, $g\in \Lp$, satisfying
 \begin{equation}\label{Mp2}
 |u(x)-u(y)|\leq d(x,y)(g(x)+g(y))\;\mu-a.e.
\end{equation}
That is,  $M_p^1= \Lp\cap \dot{M}_p^1$.  We equip $M_{p}^{1}$ with the norm
$$
 \Arrowvert u\Arrowvert_{M^{1}_{p}}= \|u\|_p + \inf_{g\textrm{ satisfies } (\ref{Mp2})}\Arrowvert g\Arrowvert_{p}.
$$
\end{dfn}

From Theorem~\ref{MN1}, we deduce that for $1\leq p\leq \infty$,
$$ M_p^1=\left\lbrace f\in \Lp: \; Nf\in \Lp\right\rbrace
$$
with equivalent norm
$$
\|f\|_{M_p^1}=\|f\|_p+\|Nf\|_p .
$$

\begin{dfn} We define the Hardy-Sobolev space $\tilMone$
as the set of all functions  $u\in \Honemax$ such that there exists
a measurable function $g\geq0$, $g\in \Lone$, satisfying
 \begin{equation}\label{Mp3}
 |u(x)-u(y)|\leq d(x,y)(g(x)+g(y))\;\mu-a.e.
\end{equation}
We equip $\tilMone$ with the norm
$$
 \| u\|_{\tilMone}= \|u^+\|_{1} + \inf_{g\textrm{ satisfies } (\ref{Mp3})}\| g\|_{1}.
$$
We have $\tilMone= \Honemax\cap \dotMone$.
\end{dfn}
Again by Theorem~\ref{MN1}, 
$$\tilMone=\left\lbrace f\in \Honemax: \; Nf\in \Lone\right\rbrace,
$$
with equivalent norm
$$
\|f\|_{\tilMone}=\|\fplus\|_{1}+\|Nf\|_1.
$$
By (\ref{fleqfplus}) and Corollary~\ref{MW}, we have
$$
\widetilde{M}_1^1\subset  M_1^1\subset W_1^1.
$$

In \cite{babe}, the authors also defined the nonhomogeneous atomic
Hardy-Sobolev spaces. Let us recall their definition.
\begin{dfn}[\cite{babe}] For $1<t\leq \infty$, we say that a function $a$ is a nonhomogeneous Hardy-Sobolev $(1,t)$-atom  if
\begin{itemize}
\item[1.] $a$ is supported in a ball $B$,
\item[2.] $\|a\|_{W^{1}_t}:=\|a\|_t+\|\grad a\|_{t}\leq \mu(B)^{-\frac{1}{t'}}$,
\item[3.] $\int a d\mu=0$.
\end{itemize}
\end{dfn}
They then define, for every $1<t\leq \infty$, the nonhomogeneous
Hardy-Sobolev space $\HStato$ as follows: $f\in \HStato$ if there
exists a sequence of nonhomogeneous Hardy-Sobolev $(1,t)$-atoms
$\{a_j\}_j$ such that $f=\sum_{j}\lambda_{j}a_{j}$ with
$\sum_{j}|\lambda_{j}|<\infty$. This space is equipped with the norm
$$
\|f\|_{\HStato}:=\inf \sum_{j}|\lambda_{j}| ,
$$
where the infimum is taken over all such decompositions.

We also recall the following comparison between these atomic Hardy-Sobolev spaces.
\begin{thm} \label{thm:comp} {\em(\cite{babe})} Let  $M$ be a complete Riemannian manifold satisfying  $(D)$ and a Poincar\'e inequality $(P_{q})$ for some $q>1$. Then $\HStato\subset  HS_{\infty,{\rm ato}}^{1}$ for every $t\geq q$ and therefore $ HS_{t_1,{\rm ato}}^{1}= HS_{t_2,{\rm ato}}^{1}$  for every $q\leq t_1,t_2\leq \infty$.
 \end{thm}

\subsection{Atomic decomposition of $\tilMone$ and comparison with $\HStato$}

As in the homogeneous case, under the Poincar\'e inequality $(P_1)$,
$\HStato\subset \tilMone$:
\begin{prop} Let $M$ be a complete Riemannian manifold satisfying $(D)$ and $(P_1)$. Let  $1<t\leq \infty$ and $a$ be
a nonhomogeneous Hardy-Sobolev $(1,t)$-atom. Then $a\in \tilMone$ with
$\|a\|_{\tilMone}\leq C_t$, the constant depending only on $t$, the
doubling constant and the constant appearing in $(P_1)$ , but not on
$a$. Consequently $\HStato\subset \tilMone$ with
$$
\|f\|_{\tilMone}\leq C_t \|f\|_{\HStato}.
$$
\end{prop}
\begin{proof} The proof follows analogously to that of Proposition~\ref{comp1},
noting that in the nonhomogeneous case every Hardy-Sobolev $(1,t)$-atom $a$ is an $\Hone$ atom and so by (\ref{Honeato-Honemax}) is in $\Honemax$
with norm bounded by a constant.
\end{proof}

Now for the converse, that is, to prove that $ \tilMone\subset
\HStato$,  we establish, as in the homogeneous case, an atomic
decomposition  for functions $f\in \tilMone$ using a
Calder\'on-Zygmund decomposition for such functions.

\begin{prop}[Calder\'on-Zygmund decomposition] \label{CZN} Let $M$ be a complete Riemannian manifold satisfying $(D)$. Let $f\in  \tilMone$, $\frac{s}{s+1}< q<1$ and $\alpha >0$. Then one can find a collection of balls $\{B_{i}\}_{i}$, functions $b_{i}\in \Wone$ and a Lipschitz function $g$ such that the following properties hold:
\begin{equation*}
f = g+\sum_{i}b_{i},
\end{equation*}
\begin{equation*}
|g(x)|+ |\grad g(x)|\leq C\alpha\quad \mbox{for}\; \mu-a.e\; x\in M,
\end{equation*}
\begin{equation*}
\supp b_{i}\subset B_{i}, \,\|b_{i}\|_1\leq C\alpha \mu(B_i)r_i,\,
\|\,b_i+ |\grad b_{i}|\,\|_{q}\leq C\alpha\mu(B_{i})^{\frac{1}{q}},
\end{equation*}
\begin{equation*}
\sum_{i}\mu(B_{i})\leq \frac{C}{\alpha}\int (\fplus+N f) d\mu,
\end{equation*}
\begin{equation*}
\mbox{and} \quad \sum_{i}\chi_{B_{i}}\leq K.
\end{equation*}
The constants $C$ and $K$ only depend on the constant in $(D)$.
\end{prop}
\begin{proof}
The proof follows the same steps as that of Proposition~\ref{CZ}. We
will only mention the changes that occur due to the nonhomogeneous
norm. Let $f\in \tilMone$,  $\frac{s}{s+1}< q<1$ and $\alpha >0$.
The first change is that we consider the open set
$$
\Omega=\{x : \Mq(\fplus +Nf)(x)>\alpha\}.
$$
 We define, as in the homogeneous case, the partition of unity $\chi_i$ corresponding to the
 Whitney decomposition $\{\underline{B_i}\}_i$ of $\Omega$, the functions $b_{i}=(f-c_i)\chi_{i}$ with
$c_i:=\frac{1}{\chi_{i}(B_{i})}\int_{B_{i}}f\chi_{i} d\mu$, and $g = f - \sum b_i$.
In addition to the previous estimates (\ref{cij1}) - (\ref{gradbiq}) for $b_i$ and $\grad b_i$,
we need here to estimate $\|b_i\|_q$.

We begin by showing that for $x\in \Omega$,
\begin{equation} \label{calp}
|c_i|\leq C \alpha
\end{equation}
 for every $i\in I_x$.
Set $\varphi_i= \gamma\frac{\chi_i}{\chi_i(B_i)}$. From the properties of $\chi_i$, in particular
since $\chi_i(B_i) \approx \mu(B_i)$, we see that we can choose $\gamma$ (independent of $i$) so
that $\varphi_i \in \Tone(y)$ and thus
$$
 |c_i|\leq \gamma^{-1}\fplus(y) \; \mbox{ for all } \; y \in B_i.
$$
Recall that
the ball $\Bibar = C_2 B_i$ has nonempty intersection with $F$. Taking $y_0 \in F \cap \Bibar$, we get, by integrating the inequality above,
$$|c_i| \leq \gamma^{-1}\left(\aver{B_i}(\fplus)^qd\mu\right)^{\frac{1}{q}}
\leq C\left(\aver{\Bibar}(\fplus)^qd\mu\right)^{\frac{1}{q}}\\
 \leq C \Mq(\fplus)(y_0)
\leq C\alpha.
$$
Combining this with (\ref{fleqfplus}), we have
$$
\|b_i\|_q \leq \left(\int_{B_i}|f - c_i|^q\right)^{\frac{1}{q}} \leq \left( \aver{\Bibar} |\fplus|^qd\mu\right)^{\frac{1}{q}}\mu(B_i)^{\frac{1}{q}} + |c_i|\mu(B_i)^{\frac{1}{q}}
\leq C\alpha\mu(B_i)^{\frac{1}{q}}.
$$
For $g$, we need to prove that $\|g\|_{\infty}\leq C\alpha$. We have
\begin{equation}
\label{eqn:g}
g=f\ind_{F}+\sum_i c_i \chi_i.
\end{equation}
For the first term we have $|f|\leq \fplus \leq \Mq(\fplus)$ at all
Lebesgue points and thus $|f\ind_F| \leq \alpha$ $\mu$-a.e. For the
second term, thanks to the bounded overlap property and
(\ref{calp}), we get the desired estimate.
\end{proof}
\begin{prop} \label{AHSN}
Let $M$ be a complete Riemannian manifold satisfying $(D)$. Let
$f\in \tilMone$. Then for all $\frac{s}{s+1}<q<1$, there is a
sequence of  $(1,\qstar)$  ($\qstar=\frac{sq}{s-q}$)  nonhomogeneous
atoms $\{a_j\}_j$,  and a sequence of  scalars $\{\lambda_j\}_j$,  such that 
$$f=\sum_j
\lambda_j a_j  \quad \mbox{ in } \Wone,
\mbox{ and } \quad  \sum |\lambda_j| \leq C_q
\|f\|_{\tilMone}.$$
Consequently, $\tilMone\subset HS^1_{\qstar, ato}$  with $\|f\|_{HS_{\qstar,{\rm ato}}^1}\leq C_q
\|f\|_{\tilMone}$.
\end{prop}
\begin{proof} Again, we will only mention the additional properties that one should
verify in comparison with the proof of Proposition \ref{AHS}.

First let us see that  (\ref{convgj}) holds in the nonhomogeneous
Sobolev space $\Wone$. We already showed convergence in the
homogeneous $\dotWone$ norm so we only need to verify convergence in
$\Lone$.  By (\ref{bi-trivial})
\begin{equation}
\label{conv-in-L1}
\|g^j-f\|_1 \leq  \sum_i\|b_i^j\|_1 \leq
 C \sum_i \int \ind_{\Bij}|f|d\mu\leq C K\int_{\Omega_j}| f|d\mu \rightarrow 0,
\end{equation}
as $j \ra \infty$.  Here we've used the properties of the
$\chi_i^j$, the bounded overlap property of the $\Bij$, the fact
that $f \in \Lone$ and that $\bigcap \Omega_j = \emptyset$ since
$\Mq(\fplus + Nf)$ is finite $\mu$-a.e.

Taking now $j \ra -\infty$, we write, by (\ref{eqn:g}), (\ref{calp}),
and the bounded overlap property
\begin{equation}
\label{minusinfty}
\int |g^j| \leq \int_{F^j}|f| + \int \sum_i |\cij| \chi_i^j
\leq \int_{\{\Mq(\fplus) \leq 2^j\}}\Mq(\fplus) + CK 2^j |\Omega^j|
\ra 0.
\end{equation}
For the functions $\ell^j= g^{j+1}-g^j$, we have
$$\|\ell^j\chi^j_k\|_{\qstar} \leq C2^j\mu(\Bjk)^{\frac{1}{\qstar}}$$
since by Proposition \ref{CZN}, $\|g^j\|_{\infty}\leq C2^j$.  This estimate also applies when we
replace $\ell^j\chi^j_k$ by the moment-free ``pre-atoms''
\begin{align*}
\elljk &:=(f-\cjk)\chi_k^j-\sum_l(f-c_l^{j+1})\chi_l^{j+1}\chi_k^j+\sum_lc_{k,l}\chi_l^{j+1} \\
& = f(1- \sum_l \chi_l^{j+1}) \chi_k^j + \cjk\chi_k^j + \sum_l c_l^{j+1}\chi_l^{j+1}\chi_k^j +\sum_lc_{k,l}\chi_l^{j+1}.
\end{align*}
The first term, involving $f$,  is $f\ind_{F_{j+1}}\chi_k^j$ which is bounded by $2^{j+1}$ since $|f| \leq \fplus \leq \Mq(\fplus)$ $\mu$-a.e. For the second and third terms, we use (\ref{calp}) and the bounded
overlap property of the $B_l^{j+1}$.  Finally, that
$$|c_{k,l}|= \left|\frac{1}{\chi_l^{j+1}(B_l^{j+1})}\int_{B_l^{j+1}}(f-c_j^{l+1})\chi_l^{j+1}\chi_k^jd\mu\right|\leq c2^j$$
follows by arguing as in the proof of (\ref{calp}), since $\frac{\chi_l^{j+1}\chi_k^j}{\chi_l^{j+1}(B_l^{j+1})}$
can be considered as a multiple of some $\varphi \in \Tone(x)$ for every $x \in \overline{B_l^{j+1}}$,
due to the fact that $|\grad \chi_k^j| \lesssim (r_k^j)^{-1} \lesssim (r_l^{j+1})^{-1}$ when $B_l^{j+1} \cap \Bjk
\neq \emptyset$.

Thus we obtain the stronger $L_\infty$ estimate
\begin{equation}
\label{Linfty}
\|\elljk\|_\infty \leq C2^j
\end{equation}
 from which we conclude,
as $\elljk$ is supported in the ball $\Bjkprime = (1 + 2c)\Bjk$, that
$\|\elljk\|_{\qstar} \leq C2^j\mu(\Bjk)^{\frac{1}{\qstar}}$.

The rest of the proof is exactly the same as that of Proposition \ref{AHS}.
\end{proof}

Now we can state the converse inclusion from Theorem~\ref{prop-Hone}:
\begin{cor}
Let $M$ be a complete Riemannian manifold satisfying $(D)$. Then
$$\Honemax(M) \subset \Honeato(M)$$
with
 $$\|f\|_{\Honeato}\lesssim \|\fplus\|_1,$$
for any choice of $t$ in the definition of $\Hone$ atoms, $1 < t \leq \infty$,
with a constant independent of $t$.
\end{cor}

\begin{proof}  Assuming $\fplus \in \Lone$ and letting
$$\Omega_j=\{x : \Mq(\fplus)(x)>2^j\},
$$
we follow the steps outlined in the proofs of Propositions~\ref{CZN} and~\ref{AHSN},
which use only the maximal function $\fplus$, while ignoring the estimates on the gradients
from the proofs of Proposition~\ref{CZ} and~\ref{AHS}, which are the only ones involving
$Nf$.  From the $L_\infty$ bound (\ref{Linfty}) we are able to obtain atoms satisfying
the conditions of Definition~\ref{Hone-atoms} with $t = \infty$, hence for every other
$t$ with uniform bounds.
\end{proof}

\textbf{\large{Conclusion:}} Let $M$ be a complete Riemannian manifold satisfying $(D)$. Then
\begin{itemize}
\item[1.] for all $\frac{s}{s+1}< q<1$,
$$\tilMone\subset HS_{\qstar,{\rm ato}}^1.$$
\item[2.] If we moreover assume  $(P_1)$, then
$$
\tilMone=\HStato
$$
for all $t>1$.
\end{itemize}

\subsection{Atomic decomposition for the Sobolev space $\Mone$}
\label{sec:Moneone}

For this we need to define new  nonhomogeneous atomic spaces
$\LStato$, where the $L$ is used to indicate that the atoms will now
be in $\Lone$ but not necessarily in $\Hone$. Let us define our
atoms.

 \begin{dfn} \label{HHSANM} For $1<t\leq \infty$, we say that a function $a$ is an $\LStato$-atom if
\begin{itemize}
\item[1.] $a$ is supported in a ball $B$;
\item[2.] $ \|\grad a\|_{t}\leq \mu(B)^{-\frac{1}{t'}}$; and
\item[3.]$\|a\|_1\leq \min(1,r(B))$.
\end{itemize}
We then  say that $f$ belongs to $\LStato$ if there exists a
sequence of $\LStato$-atoms $\{a_{j}\}_{j}$  such that
$f=\sum_{j}\lambda_{j}a_{j}$ in $\Wone$, with
$\sum_{j}|\lambda_{j}|<\infty$. This space is equipped with the norm
$$
\|f\|_{\LStato}=\inf \sum_{j}|\lambda_{j}|,
$$
where the infimum is taken over all such decompositions.
\end{dfn}

\begin{rem} As discussed previously, condition $3$ in Definition~\ref{HHSANM} is a substitute for
the cancellation condition $3$ in Definition~\ref{HHSA}.  Assuming a
Poincar\'e inequality ($P_t$), $\LStato$-atoms corresponding to
small balls (with $r(B)$ bounded above) can be shown (see \cite{D},
Appendix B) to be elements of Goldberg's local Hardy space (defined
by restricting the supports of the test functions in
Definition~\ref{Hone-max} to balls of radii $r < R$ for some fixed $R$ - see \cite{Stein},
Section III.5.17), so that
$\LStato$ is a subset of the ``localized'' space $H_{1,{\rm loc}}$.
\end{rem}

As in the homogeneous case, under the Poincar\'e inequality $(P_1)$,
$\LStato\subset \Mone$:
\begin{prop} Let $M$ be a complete Riemannian manifold satisfying $(D)$ and $(P_1)$.
Let  $1<t\leq \infty$ and $a$  be an $\LStato$-atom. Then $a\in \Mone$ with
$\|a\|_{\Mone}\leq C_t$, the constant $C$ depending only on $t$, the
doubling constant and the constant appearing in $(P_1)$, and
independent of $a$.

Consequently $\LStato\subset \Mone$ with
$$
\|f\|_{\Mone}\leq C_t \|f\|_{\LStato}.
$$
\end{prop}
\begin{proof} The proof follows analogously to that of Proposition \ref{comp1},
noting that we can use Remark \ref{remcomp1} thanks to property 3 in
Definition~\ref{HHSANM}, and that this property also implies every
atom $a$ is in $\Lone$.
\end{proof}

Now for the converse, that is, to prove that $\Mone\subset \LStato$,
we  again establish  an atomic decomposition  for functions $f\in
\Mone$.  In order to do that we must introduce an equivalent 
maximal function $\fstar$, which is a variant of the one
originally defined by Calder\'on \cite{calderon} and denoted by
$N(f,x)$ (here we are only defining it in the special case $q = 1$
and $m=1$, where for $x$ a Lebesgue point of $f$, the constant
$P(x,y)$ in Calder\'on's definition is equal to $f(x)$, and we are
allowing for the balls not to be centered at $x$):

\begin{dfn} Let $f\in \Loneloc(M)$.  Suppose $x$ is a Lebesgue point of
$f$, i.e.
$$\lim_{r \ra 0}\;\ \aver{B(x,r)}|f(y)-f(x)| d\mu(y) = 0.
$$
We define 
$$
\fstar(x):=\sup_{B: \,x\in B}\frac{1}{r(B)}\aver{B}|f(y)-f(x)| d\mu(y).
$$
Then $\fstar$ is defined $\mu$-almost everywhere.
\end{dfn}

We now show the equivalence of $\fstar$ and $Nf$.
As discussed in the Introduction, the following Proposition was proved in \cite{devsha} (see also \cite{miyachi})
in the Euclidean case:
\begin{prop} \label{molN}
Let $M$ be a complete  Riemannian manifold satisfying $(D)$. Then, there exist constants $C_1,\, C_2$
such that for all $f\in \Loneloc(M)$
 $$
 C_1 Nf\leq \fstar\leq C_2 Nf
 $$
pointwise $\mu$-almost everywhere.
\end{prop}

\begin{proof} Let $f\in \Loneloc$ and  $x$ be a Lebesgue point of $f$, so that there exists a
sequence of balls $B_n = B(x,r_n)$ with $r_n \ra 0$ and
$f_{B_n} \ra f(x)$. Given a ball $B$ containing $x$, take $n$ sufficiently large so that
$B_n \subset B$.
Since $x \in B$, there is a smallest $k \geq 1$ such that
$2^k B_n = B(x,2^kr_n) \supset B$, and for this $k$ we have $2^k r_n \leq 4 r(B)$, so
\begin{align*}
|f_B-f_{B_n}(x)|
&\leq \aver{B}|f - f_{2^kB_n}|d\mu + \sum_{j = 1}^k |f_{2^jB_n} - f_{2^{j-1}B_n}|\\
&\leq \frac{\mu(2^kB_n)}{\mu(B)}\aver{2^kBn}|f- f_{2^kB_n}|d\mu  + \sum_{j = 1}^k \frac{\mu(2^jB_n)}{\mu(2^{j-1}B_n)} \aver{2^jB_n} |f-f_{2^jB_n}|d\mu
\\
&\leq 2C_{\rm{(D)}}^2  \sum_{j = 1}^k 2^j r_n Nf(x)
\\
&\leq 16C_{\rm{(D)}}^2 r(B) Nf(x).
\end{align*}
Taking the limit as $n \ra \infty$, we see that $|f_B - f(x)| \leq C
r(B) Nf(x)$ so that
$$
\aver{B}|f(y)-f(x)|d\mu(y) \leq
\aver{B}|f(y)-f_B|d\mu(y)+|f_B-f(x)| \leq Cr(B) Nf(x).
$$
Dividing by $r(B)$ and taking the supremum over all balls $B$
containing $x$, we conclude that
$\fstar(x)\leq C Nf(x)$.

For the converse, again take any Lebesgue point $x$ and let $B$ be a ball
containing $x$.  Writing $|f(y) - f_B|
\leq |f(y) - f(x)| + |\aver{B}f - f(x)|$, we have
$$
\aver{B}|f(y)-f_B|d\mu(y)
\leq  2\aver{B} |f(y)-f(x)|d\mu(y) \leq  2 r(B)  \fstar(x).
$$
Taking the supremum over all balls $B$ containing $x$, we deduce
that $ Nf(x) \leq 2\fstar(x)$.
\end{proof}

\begin{prop}[Calder\'on-Zygmund decomposition] \label{CZNM} Let $M$ be a complete
Riemannian manifold satisfying $(D)$. Let $f\in \Mone$, $\frac{s}{s+1}< q<1$
and $\alpha >0$. Then one can find a collection of balls $\{B_{i}\}_{i}$, functions
$b_{i}\in \Wone$ and a Lipschitz function $g$ such that the following properties
hold:
\begin{equation*}
f = g+\sum_{i}b_{i},
\end{equation*}
\begin{equation}
|g(x)|+ |\grad g(x)|\leq C\alpha\quad \mbox{for} \; \mu-a.e\; x\in M, \label{egnm}
\end{equation}
\begin{equation}
\supp b_{i}\subset B_{i}, \,\|b_{i}\|_1\leq C\alpha \mu(B_i)r_i,\, \|\,b_i+ |\grad b_{i}|\,\|_{q}\leq C\alpha\mu(B_{i})^{\frac{1}{q}}, \label{ebnm}
\end{equation}
\begin{equation}
\sum_{i}\mu(B_{i})\leq \frac{C_q}{\alpha}\int (|f|+ Nf) d\mu,
\label{eBn}
\end{equation}
\begin{equation}
\mbox{and} \quad \sum_{i}\chi_{B_{i}}\leq K \label{rbnm}.
\end{equation}
The constants $C$ and $K$ only depend on the constant in $(D)$.
\end{prop}

\begin{proof}
The proof follows the same  steps as that of Propositions \ref{CZ}
and \ref{CZN}. Again we will only mention the changes that occur.
Let $f\in \Mone$,  $\frac{s}{s+1}<q<1$ and $\alpha >0$.  By
Proposition~\ref{molN}, we have $\fstar \in \Lone$ with norm
equivalent to $\|Nf\|_1$.  Thus if we consider the open set
$$
\Omega=\{x:\, \Mq(|f|+\fstar)(x)>\alpha\},
$$
its Whitney decomposition $\{B_i\}_i$, and the corresponding partition of unity $\{\chi_i\}_i$,
we get immediately (\ref{rbnm}) and (\ref{eBn})
by the bounded overlap property and the boundedness of the maximal function in $L_{1/q}$.

We again define $b_i = (f-c_i)\chi_i$ but this time we set
$c_i=f(x_i)$
for some $x_i \in \Bibar$ chosen as follows.  Recall that $\Bibar = 4B_i$ contains
some point $y$ of $F = M \setminus \Omega$ so that
\begin{equation}
\label{averBibar}
\aver{\Bibar} |f|^q \leq \Mq(f)^q(y) \leq \alpha^q
\end{equation}
as well as
\begin{equation}
\label{averBibarfstar}
\aver{\Bibar} (\fstar)^q \leq \Mq(\fstar)^q(y) \leq \alpha^q.
\end{equation}
Let
$$E_i = \{x \in \Bibar: x \mbox{ is a Lebesgue point of $f$ and $|f|^q$, and } |f(x)| \leq 2\alpha\}.$$
We claim that
$$\mu(E_i) \geq (1 - 2^{-q}) \mu(\Bibar).$$
Otherwise we would have $\mu(\Bibar \setminus E_i) > 2^{-q} \mu(\Bibar)$ and
so, since $f$ and $|f|^q$ are locally integrable and the set of points which are not their
Lebesgue points has measure zero,
$$\int_{\Bibar \setminus E_i} |f|^q \geq (2\alpha)^q\mu(\Bibar \setminus E_i) > \alpha^q\mu(\Bibar),$$
contradicting (\ref{averBibar}).

Now we claim that for an appropriate constant $c_q$ (to be chosen independent of $i$ and $\alpha$),
there exists a point $x_i \in E_i$ with
\begin{equation}
\label{fstarxi}
\fstar(x_i) \leq c_q\alpha.
\end{equation}
Again, suppose not.  Then we have,
by (\ref{averBibarfstar}),
$$(c_q\alpha)^q\mu(E_i) \leq \int_{E_i} (\fstar)^q d\mu \leq \alpha^q\mu(\Bibar),$$
implying that $\mu(E_i) \leq c_q^{-q}\mu(\Bibar)$.  Taking $c_q > (1 - 2^{-q})^{-1/q}$,
we get a contradiction.

Thanks to our choice of $x_i$, we now have
$$|c_i| = |f(x_i)| \leq 2\alpha$$
and
$$
\|b_i\|_1 \leq C\int_{B_i}|f(y)-f(x_i)| d\mu(y)
\leq C\mu(B_{i})r_i \fstar(x_i)
\leq C c_q r_i \alpha \mu(B_i).
$$
Moreover for $\|b_i\|_q$, one has, by (\ref{averBibar}),
$$
 \|b_i\|_q \leq C \left( \int_{B_i} |f-c_i|^qd\mu\right)^{\frac{1}{q}}
 \leq C\left(\int_{B_i}|f|^qd\mu\right)^{\frac{1}{q}}+C2\alpha\mu(B_i)^{\frac{1}{q}}
\leq C\alpha \mu(B_i)^{\frac{1}{q}}.
$$
Finally, for $\grad b_i$, we can estimate the $L_1$ norm by
\begin{align}
\|\grad b_i \|_1 &\leq  \|(f-c_{i})\grad \chi_i| \|_1+\|(\grad f) \chi_i \|_1
\nonumber \\
&\leq  \int_{B_{i}}|f(x)-f(x_i)| |\grad\chi_{i}(x)|d\mu(x) +\int_{B_i}|\grad f| d\mu
\nonumber\\
&\leq C\mu(B_i)\fstar(x_i)+\int_{B_i}|\grad f| d\mu
\nonumber\\
&\leq C c_q \alpha \mu(B_i) + \int_{B_i}|\grad f| d\mu,
\label{gradbi1'}
\end{align}
showing (since $|\grad f|$ in $L_1$ by Proposition~\ref{nabN}) that $b_i \in \Wone$,
and the $L_q$ norm by
\begin{align*}
\|\grad b_i \|^q_q &\leq  \|(f-c_{i})\grad \chi_i| \|^q_q+\|(\grad f) \chi_i \|_q^q
\\
&\leq  \mu(B_i)^{1-q}\left(\int_{B_{i}}|f(x)-f(x_i)| |\grad\chi_{i}(x)|d\mu(x)\right)^{q}+\int_{B_i}|\grad f|^qd\mu
\\
&\leq C\mu(B_i)\fstar(x_i)^q+\int_{\Bibar}|Nf|^qd\mu\\
&\leq C(c_q \alpha)^q\mu(B_i) +\int_{\Bibar}|\fstar|^qd\mu
\\
&\leq C \alpha^q \mu(B_i),
\end{align*}
where we used Propositions~\ref{nabN} and~\ref{molN}, and (\ref{averBibarfstar}).
Taking the $1/q$-th power on both sides, we get (\ref{ebnm}).

It remains to prove (\ref{egnm}). First note that $\|g\|_{\infty}\leq C\alpha$ since
$$
g=f\ind_{F}+\sum_i c_i \chi_i
$$
and for the first term, by the Lebesgue differentiation theorem, we
have $|f\ind_F| \leq  \Mq(f)\ind_F \leq \alpha$ $\mu$-a.e., while
for the second term, thanks to the bounded overlap property and
$|c_i| \leq 2\alpha$, we get the desired estimate.

Now for the gradient, we write, as in (\ref{gradg}),
$$\grad g = \ind_{F}(\grad f) -\sum_{i}(f-f(x_i))\grad\chi_{i}.
$$
Again we have, by Propositions~\ref{nabN} and~\ref{molN}, that $\ind_{F}(|\grad f|)
\leq C\ind_{F}(Nf) \leq C\ind_{F}(\fstar) \leq C\alpha\;\mu -$a.e. Let
$$h=\sum_{i}(f-f(x_i))\grad\chi_{i}.$$
We will show $|h(x)|\leq C\alpha$ for all $x\in M$. Note first that
the sum defining $h$ is locally finite on $\Omega$ and vanishes on
$F$. Then take  $x\in \Omega$ and a Whitney ball $B_{k}$ containing
$x$. As before, since $\displaystyle \sum_{i}\grad\chi_{i}(x)=0$, we
can replace $f(x)$ in the sum by any constant so
$$h(x)=\sum_{i\in I_x} \left(f(x_k) -f(x_i) \right)\grad
\chi_{i}(x).$$
Recall that for all $i,k\in I_{x}$, by the construction of  the Whitney collection, the balls $B_i$ and $B_k$ have equivalent radii and  $B_i \subset 7B_k$.
Thus
\begin{align}
\label{fxk-fxi}
|f(x_k)-f(x_i)| & \leq |f_{7B_k} - f(x_k)| +  |f_{7B_k} - f(x_i)|\\
& \leq \aver{7B_k}|f - f(x_k)| d\mu + \aver{7B_k}|f - f(x_i)| d\mu \nonumber \\
& \leq 7r_k (f^\star(x_k) + f^\star(x_i)) \leq 14 r_k c_q \alpha,\nonumber
\end{align}
by (\ref{fstarxi}).
Therefore we again get the estimate (\ref{h}).

\end{proof}
\begin{prop} \label{AHSN2}Let $M$ be a complete Riemannian manifold satisfying $(D)$. Let $f\in \Mone$. Then for all $\frac{s}{s+1}< q<1$,  there is a sequence of  $LS^{1}_{\qstar,{\rm ato}}$-atoms $\{a_j\}_j$ ($\qstar=\frac{sq}{s-q}$),  as in Definition \ref{HHSANM},
and a sequence of scalars $\{\lambda_j\}_j$,  such that $$f=\sum_j \lambda_j a_j \quad \mbox{ in } \Wone,
 \mbox{ and } \quad
\sum |\lambda_j| \leq C_q \|f\|_{\Mone}.$$
Consequently, $\Mone\subset HS^1_{\qstar, ato}$ with
$\|f\|_{LS_{\qstar,{\rm ato}}^1}\leq C_q \|f\|_{\Mone}$.
\end{prop}

\begin{proof} Here as well we will only mention the additional properties that one should verify in comparison with  Proposition \ref{AHS} and \ref{AHSN}.
We use the Calder\'on-Zygmund decomposition (Proposition~\ref{CZNM}) above with $\Omega^j$ corresponding to
$\alpha = 2^j$, and denote the resulting functions by $g^j$ and $b_i^j$, recalling that for the definition
of the constant $\cij$ we have $\cij = f(x_i^j)$ for a specially chosen point $x_i^j \in \Bijbar$.

First let us see that  $g^j \rightarrow f$ in $W_1^1$. For the
convergence in $\Lone$ we just repeat (\ref{conv-in-L1}) and
(\ref{minusinfty}) from the nonhomogeneous case, replacing $\fplus$
by $|f|$. For the convergence in $\dot{W}_1^1$, we can estimate
$\sum_i\|\grad b_i^j\|_1$ exactly as in (\ref{gradbij}), using (\ref{gradbi1'}) instead
of (\ref{gradbi1}), and replacing
$Nf$ by $\fstar$ and $\Mq(Nf)$ by $\Mq(|f| + \fstar)$. This gives
$\grad g^j \rightarrow \grad f$ in $\Lone$ as $j \ra \infty$.  For
the convergence of $\grad g^j$ to $0$ as $j \ra -\infty$, we imitate
(\ref{gradgj}) and (\ref{gradgj1}), using (\ref{gradg}) and
(\ref{h}) with $\fstar$ and our new choice of $\cij$.

We define the functions $\ell^j = g^{j+1} - g^j$ as in Proposition \ref{AHS} but this time
we just use
$$\elljk := \ell^j\chi^j_k$$
for the ``pre-atoms'', since we no longer need to have the moment condition $\int \elljk = 0$
(see Remark~\ref{rem-otheratoms}).  From the $L_\infty$ bounds (\ref{egnm}) on $g^j$ and $\grad g^j$ in
Proposition \ref{CZNM}, we immediately get
$$\|\elljk\|_{1} \leq C2^j\mu(\Bjk)$$and
$\|\,|\grad \ell^j|\chi_k^j\|_{\qstar}\leq C 2^j
\mu(\Bjk)^{1/\qstar}$.  We need a similar estimate on $\|\ell^j|\grad \chi_k^j| \|_{\qstar}$ in
order to bound $\|\grad \elljk\|_{\qstar}$.
As in (\ref{gradelljk}), write
$$
r_k^j \; \left(\aver{\Bjk}| \ell^j\grad \chi_k^j|^{\qstar}d\mu\right)^{1/\qstar}
 \leq C \left(\aver{\Bjk}\Big(\sum_{i} \ind_{B_{i}^j}|f-\cij|
+\sum_{l}
\ind_{B_{l}^{j+1}}|f-c_{l}^{j+1}|\Big)^{\qstar} d\mu\right)^{1/\qstar}
$$
Expanding $|f - \cij| =  |f- f_{\Bjk} + f_{\Bjk} - \cjk + \cjk - \cij|$ and using the bounded
overlap property of the balls, the Sobolev-Poincar\'e inequality (\ref{PNq}), Proposition~\ref{molN},
and properties (\ref{fstarxi}) and (\ref{fxk-fxi}) of the constants $\cij = f(x^j_i)$, we have
for the integral of the first sum on the right-hand-side:
\begin{align*}
\left(\aver{\Bjk}\Big(\sum_{i} \ind_{B_{i}^j}|f-\cij|\Big)^{\qstar} d\mu\right)^{1/\qstar}
&\leq K \left(\aver{\Bjk}|f- f_{\Bjk}|^{\qstar} d\mu \right)^{1/\qstar}  + K|f_{\Bjk} - \cjk| \\
& + \left(\aver{\Bjk}\Big(\sum_{\Bij \cap \Bjk \neq \emptyset}
\ind_{\Bij}|\cjk - \cij| \Big)^{\qstar} d\mu \right)^{1/\qstar} \\
& \leq CK r^j_k \left(\aver{\Bjkbar} (Nf)^q\right)^{1/q} + K r^j_k \fstar(x^j_k) + CK r^j_k 2^j\\
& \leq CK r^j_k 2^j.
\end{align*}
The analogous estimate holds for the integral of the second sum, in $l$, since as pointed out
previously, when $B_l^{j+1} \cap \Bjk \neq \emptyset$ we have that $r_l^{j+1} \leq cr^j_k$.
This gives
$$
\|\grad \elljk\|_{\qstar}\leq \gamma  2^j \mu(\Bjkprime)^{1/\qstar},
$$
as desired. The rest of the proof follows in the same way as that of
Propositions \ref{AHS} and \ref{AHSN}.
\end{proof}

\textbf{\large{Conclusion:}} Let $M$ be a complete Riemannian manifold satisfying $(D)$. Then
\begin{itemize}
\item[1.] for all $\frac{s}{s+1}<q<1$,
$$\Mone\subset LS_{\qstar,{\rm ato}}^1.$$
\item[2.] If moreover we assume  $(P_1)$, then
$$
\Mone=\LStato
$$
for all $t>1$.
\end{itemize}

\section{Comparison between $\dotMone$ and Hardy-Sobolev spaces defined in
terms of derivatives} \label{sec:comp}

\subsection{Using a maximal function definition}

In the Euclidean case, the homogeneous Hardy-Sobolev space
$\dot{HS}^1$ consists of all locally integrable functions $f$ such
that $\grad f \in \Hone(\Rn)$ (i.e.\ the weak partial derivatives
$D_jf=\frac{\partial f}{\partial x_j}$ belong to the real Hardy
space $\Hone(\Rn)$). In \cite{miyachi}, it was proved that this
space is nothing else than $\left\lbrace f\in \Loneloc(\Rn) : \, Nf
\in \Lone\right\rbrace$, which also coincides with the Sobolev space
$\dotMone $ (\cite{KS}).

Does this theory extends to the case of Riemannian manifolds? If
this is the case, which hypotheses should one assume on the geometry
of the manifold? We proved an atomic  characterization of $\dotMone$
but we would like
to clarify the relation with
Hardy-Sobolev spaces defined using maximal functions.

\begin{dfn}
\label{dfn:HS1max} We define the (maximal) homogeneous Hardy-Sobolev space $\HSonemax$ as follows:
$$ \HSonemax:= \left\lbrace f\in \Loneloc(M):\, (\grad f)^+ \in \Lone\right\rbrace
$$
where $\grad f$ is the distributional gradient, as defined in {\rm(\ref{graddist})},
and the corresponding maximal function is defined, analogously to {\rm(\ref{def:fplus})}, by
$$(\grad f)^+(x):=\sup \left| \int f\;(\langle \grad \varphi, \bPhi \rangle +
\varphi \Div \bPhi) \, d\mu \right|,
$$
where the supremum is taken over all pairs $\varphi \in \Tone(x)$, $\bPhi \in C^1_0(M, TM)$ such that
 $$
\|\bPhi\|_{\infty} \leq 1 \quad \mbox{and }  \|\Div \bPhi\|_{\infty}\leq \frac{1}{r}
 $$
for the radius $r$ of the same ball $B$ containing $x$ for which $\varphi$ satisfies (\ref{Tone}).
We equip this space with the semi-norm
$$ \|f\|_{\HSonemax}= \|(\grad f)^+\|_1.$$
\end{dfn}
Note that in case both $\varphi$ and $\bPhi$ are smooth, the quantity $\langle \grad \varphi, \bPhi \rangle +
\varphi \Div \bPhi$ represents the divergence of the product $\varphi\bPhi$, so the definition coincides with
that of the maximal function $M^{(1)}f$ given in \cite{ART} for the case of domains in $\Rn$, but here
we want to allow for the case of Lipschitz $\varphi$.

\begin{prop}Let $f\in \HSonemax$. Then $\grad f$, initially defined by
(\ref{graddist}), is given by an $\Lone$ function and satisfies
$$
|\grad f| \leq C(\grad f)^+ \qquad   \mu-{\rm a.e.}
$$
Consequently,
$$ \HSonemax \subset \dotWone
$$
with
$$
\|f\|_{\dotWone} \leq C \|f\|_{\HSonemax}.
$$
\end{prop}

\begin{proof}  We follow the ideas in the proof of Proposition~\ref{nabN}.  Let $\Omega$
be any open subset of $M$ and consider the total variation of $u$ on $\Omega$, defined by
$$|Df|(\Omega):= \sup \left| \langle \grad f, \bPhi \rangle \right|,
$$
where the supremum is taken over all vector fields $\bPhi \in C^1_0(\Omega,TM)$
with $\|\bPhi\|_\infty \leq 1$.  For such a vector field $\bPhi$, take $r>0$ sufficiently small so
that $\|\Div \bPhi\|_\infty \leq r^{-1}$ and $\dist(\supp(\bPhi),M\setminus\Omega) > 12r$.  As in the proof
of Proposition~\ref{nabN}, take a
collection of balls $B_i = B(x_i, r)$ with $6B_i$ having bounded overlap (with a constant $K$ independent of $r$), covering $M$, and
a Lipschitz partition of unity $\{\varphi_i\}_i$ subordinate to $\{6B_i\}_i$, with $0 \leq \varphi_i \leq 1$
and $|\grad \varphi_i| \leq r^{-1}$.  Then for all $x \in B_i$, $\varphi_i/\mu(B_i) \in \Tone(x)$,  so
$$\left| \int f [\langle \grad \varphi_i, \bPhi\rangle + \varphi_i \Div\bPhi] d\mu \right| \leq (\grad f)^+(x)\mu(B_i).$$
Hence
$$\left| \int f [\langle \grad \varphi_i, \bPhi\rangle + \varphi_i \Div\bPhi] d\mu \right| \leq \int_{B_i}
(\grad f)^+(x)d\mu.$$
Summing up over $i$ such that $6B_i \subset \Omega$, by the choice of $r$ we still get $\sum \varphi_i = 1$ on the support of $\bPhi$,
hence $\sum \grad \varphi_i = 0$, so using the bounded overlap of the balls we have
$$\left| \int f\; \Div \bPhi \, d\mu \right| \leq \sum_{\{i: 6B_i \subset \Omega\}}\int_{B_i}(\grad f)^+ d\mu
\leq K \int_\Omega(\grad f)^+ d\mu \leq K \|(\grad f)^+\|_1 < \infty.$$
The rest of the proof proceeds as in the proof of Proposition~\ref{nabN}, replacing
$Nf$ by $(\grad f)^+$.
\end{proof}

\begin{prop} Let $f\in \Loneloc$. Then at every point of $M$,
$$
(\grad f)^+ \leq  Nf.
$$
Consequently,
$$ \dotMone  \subset \HSonemax
$$
with
$$
\|f\|_{\HSonemax} \leq C \|f\|_{\dotMone}.
$$
\end{prop}

\begin{proof}
Let $f\in \Loneloc$ and $x\in M$. Take $\varphi \in \Tone(x)$, $\bPhi \in C^1_0(M, TM)$ as in Definition~\ref{dfn:HS1max}. Then
$$\int (\langle \grad \varphi, \bPhi \rangle + \varphi \Div \bPhi)d\mu = 0$$
so we can write
\begin{align*}
\left| \int f (\langle \grad \varphi, \bPhi \rangle + \varphi \Div \bPhi) d\mu\right|&=\left| \int (f-f_B) (\langle \grad \varphi, \bPhi \rangle + \varphi \Div \bPhi)d\mu \right|
\\
&\leq \frac{1}{r\mu(B)}\int|f-f_B|d\mu
\\
& \leq Nf(x).
\end{align*}
\end{proof}

We would like to prove the reverse inclusion.  However, this would require some tools such
as Lemma 6 in \cite{KS} or Lemma 10 in \cite{ART} (solving $\Div \Psi = \phi$ with $\Psi$ having compact support) which are particular to $\Rn$.

Another possible maximal function we can use, following the ideas in \cite{kintuo} (see Section 4.1),
is given by
\begin{dfn}
\label{dfn:gradmax}
$$\cM^*(\grad f)(x):=\sup_j |\grad f_{r_j}|
$$
with the ``discrete convolution'' $f_{r_j}$ defined as in (\ref{dfn:ur}), corresponding
to an enumeration of the positive rationals $\{r_j\}_j$, where for each $j$ we have
a covering of $M$ by balls $\{B^j_i\}_i$  of radius $r_j$, and a partition of unity
$\varphi^j_i$ subordinate to this covering.
\end{dfn}

We have already shown in the proof of Proposition~\ref{nabN} (see (\ref{urmaxN})) that
\begin{lem} Let $f\in \Loneloc$. Then at $\mu$-almost every point of $M$,
$$
\cM^*(\grad f) \leq  Nf.
$$
\end{lem}

\subsection{Derivatives of molecular Hardy spaces}
\label{comp2}

As noted in the previous section, on a manifold, obtaining a decomposition with atoms of compact support from a maximal function definition is not obvious.  In \cite{AMR}, the authors considered instead Hardy spaces generated by molecules.
We begin by recalling their definition of $\Honemol(\wedge^1 T^*M)$ (a special case with $N = 1$ of
$H^1_{{\rm mol}, N}(\wedge T^*M)$ in Definition 6.1 of
\cite{AMR}, where we have dropped the superscript $1$ for convenience).
If in addition the heat kernel on $M$ satisfies Gaussian upper bounds, this space coincides
with the space $H^1(\wedge T^*M)$, which also has a maximal function characterization (see \cite{AMR},
Theorem 8.4).

A sequence of non-negative Lipschitz functions
$\{\chi_k\}_k$  is said to be (a partition of unity) adapted to a ball $B$ of
radius $r$  if $\supp \chi_0\subset 4B$, $\supp \chi_k\subset 2^{k+2}B\setminus 2^{k-1}B$ for all $k\geq 1$,
\begin{equation}
\label{gradchik}
\|\grad \chi_k\|_{\infty}\leq C 2^{-k}r^{-1}
\end{equation} and
$$\sum_k \chi_k=1 \; \textrm {on } M.$$
A $1$-form $a \in L^2(\wedge^1 T^*M)$ is called a $1$-molecule if $a=db$ for some $b\in L_2(M)$ and
there exists a ball $B$ with radius $r$, and a partition of unity $\{\chi_k\}_k$ adapted to $B$, such that for all $k\geq 0$
\begin{equation}
\label{chikaL2norm}
\|\chi_k a\|_{L^2(\wedge^1 T^*M)}\leq 2^{-k}(\mu(2^kB))^{-1/2}
\end{equation}  and
$$\|\chi_k b\|_{2}\leq 2^{-k}r(\mu(2^kB))^{-1/2}.$$
Summing in $k$, this implies that
$\|a\|_{L^2(\wedge^1 T^*M)}\leq 2(\mu(B))^{-1/2}$ and $\|b\|_{L^2}\leq 2r(\mu(B))^{-1/2}$. Moreover,
there exists a constant $C'$, depending only on the doubling constant in ($D$), such that
\begin{equation}
\label{bL2norm}
\|\ind_{2^{k+2}B\setminus 2^{k-1}B}b\|_{2}\leq \left\|\sum_{l = k - 3}^{k + 3}\chi_l b\right\|_{2}
\leq C'r2^{-k}(\mu(2^{k+2}B))^{-1/2}.
\end{equation}

\begin{dfn}[\cite{AMR}]\label{HSmol} We say that $f\in \Honemol(\wedge^1 T^*M)$
if there is a sequence $\{\lambda_j\}_j\in \ell^1$ and a sequence of
$1$-molecules $\{a_j\}_j$ such that
$$
f=\sum_j\lambda_ja_j
$$
in $\Lone(\wedge^1T^*M)$,  with  the norm defined by
$$
\|f\|_{\Honemol(\wedge^1 T^*M)}=\inf \sum_j|\lambda_j|.$$
Here the infimum is taken over all such decompositions.
The space $\Honemol(\wedge^1 T^*M)$ is a Banach space.
\end{dfn}

 \begin{prop}  Let $M$ be a complete Riemannian manifold satisfying $(D)$ and $(P_1)$.
 We then  have
\begin{equation}
\label{Honemol=dHSone}
\Honemol(\wedge^1 T^*M)= d(\dot{HS}^1_{2,{\rm ato}}(M)).
\end{equation}
Moreover
 $$\|g\|_{\Honemol(\wedge^1 T^*M)}\sim \inf_{df = g} \|f\|_{\dot{HS}^1_{2,{\rm ato}}(M)}.$$
 Consequently, in this case we have an atomic decomposition for $\Honemol(\wedge^1 T^*M)$ (this was already proved in
\cite{AMR}, after Theorem 8.4).
\end{prop}

\begin{rem}
As pointed out in Remarks~\ref{remcomp1} and ~\ref{rem-otheratoms}, we can define
the atomic Hardy-Sobolev space $\dot{HS}^1_{2,{\rm ato}}(M)$ by using $(1,2)$-atoms satisfying
condition $3^{\prime\prime}$ of Remarks~\ref{atom-rems} instead of
condition $3$ of Definition~\ref{HHSA}.  As will be seen from the proof below, if we
restrict ourselves to this kind of atoms we do not require the hypothesis
$(P_1)$ for (\ref{Honemol=dHSone}).  Under the assumption $(P_1)$, we actually get the stronger conclusion
$$\Honemol(\wedge^1 T^*M)=d(\dot{HS}_{2,{\rm ato}}^1)=d(\dotHStato)=d(\dotMone)$$
for all $t > 1$.
\end{rem}

 \begin{proof}
Take $f\in \dot{HS}^1_{2,{\rm ato}} $. There exists a sequence
$\{\lambda_j\}_j\in \ell^1$ and $(1,2)-$atoms $b_j$ such that
$f=\sum_j\lambda_j b_j$ in $\dot{W}_1^1$. This means
$\sum_{j}\lambda_{j}\grad b_{j}$ converges in  $\Lone$ to $\grad f$,
and by the isometry between the vector fields and the $1$-forms, we
have $df=\sum_j \lambda_j db_{j}$ in $\Lone(\wedge^1T^*M)$.

We claim that $a_j = db_{j}$ are $1$-molecules.  Indeed, fix $j$, take $B_{j}$ to be the ball containing the
support of $b_j$ and let $\{\chi_{j}^k\}_k$ be a partition of unity adapted to $B_{j}$.
Then
 $$
\|\chi_{j}^0a_{j}\|_2\leq \|db_{j}\|_2 = \|\grad b_{j}\|_2 \leq \frac{1}{\mu(B_{j})^{\frac{1}{2}}}$$
 and by condition $3^{\prime\prime}$ of Remarks~\ref{atom-rems} (alternatively
condition $3$ of Definition~\ref{HHSA} and $(P_1)$) we get
 $$
 \|\chi_{j}^0b_{j}\|_2\leq \|b_{j}\|_2\leq r_{j}\frac{1}{\mu(B_{j})^{\frac{1}{2}}}.$$
  For $k\geq 1$, there is nothing to do since $\supp b_{j}\subset B_{j}$ and
$\supp\chi_{j}^k\subset 2^{k+2}B_{j}\setminus 2^{k-1}B_{j}\subset (B_{j})^c$.
Consequently,  $df\in \Honemol(\wedge^1 T^*M)$ with $\|df\|_{\Honemol(\wedge^1 T^*M)}\leq \sum_j|\lambda_j|$.
Taking the infimum over all such decompositions, we get $\|df\|_{\Honemol(\wedge^1 T^*M)}\leq\|f\|_{\dot{HS}^1_{2,{\rm ato}} }$.

Now for the converse, let  $g\in \Honemol(\wedge^1 T^*M)$. Write
$$g=\sum_j\lambda_j a_j:= \sum_j \lambda_j db_j$$
where $\sum_j |\lambda_j| < \infty$, for every $j$, $a_j$ is a
$1$-molecule associated to a  ball $B_j$, and the convergence is in
$\Lone$.  Let $\{\chi_j^k\}_k$ be the partition of unity adapted to
$B_j$. Then
$$g=\sum_j\lambda_j \sum_kdb_j\chi_j^k= \sum_j\lambda_j d( \sum_kb_j\chi_j^k)=\sum_j\lambda_j \sum_k d(b_j\chi_j^k)$$
since the sum is locally finite and $\sum_k \chi_j^k=1$.

We claim that for every $j,\,k$,  $\beta_j^k:=2^{k-1}\gamma b_j\chi_j^k$, with $\gamma$ a constant to be determined,  satisfies properties $1,2$ and $3^{\prime\prime}$ (see  Definition~\ref{HHSA} and  Remarks~\ref{atom-rems})
of a $(1,2)$-homogeneous Hardy-Sobolev atom.
Indeed, $\beta_j^k$ is supported in the ball $2^{k+2}B_j$ with
$$\|\beta_j^k\|_2\leq 2^{k-1}\gamma\frac{2^{-k}r_j}{\mu(2^kB_j)^{\frac{1}{2}}} \leq \frac{2^{k+2}r_j}{\mu(2^{k+2}B_j)^{\frac{1}{2}}}$$
for an appropriate choice of $\gamma$ depending only on the doubling constant in ($D$).
Furthermore, by (\ref{chikaL2norm}), (\ref{gradchik}), and (\ref{bL2norm}),
\begin{align*}
\|\grad \beta_j^k\|_2 & = 2^{k-1} \gamma \|d(b_j\chi_j^k)\|_2\\
& \leq 2^{k-1} \gamma( \|a_j\chi_j^k\|_2+ \|b_j d\chi_j^k\|_2) \\
& \leq 2^{k-1} \gamma(2^{-k}(\mu(2^kB_j))^{-1/2} +
C2^{-k}r_i^{-1} \|\ind_{2^{k+2}B_j\setminus 2^{k-1}B_j}b_j \|_2)\\
&\leq  \mu(2^{k+2}B_j)^{-1/2}.
\end{align*}
Here we again chose $\gamma$ conveniently, depending only on the doubling constant, and
used the fact that $k \geq 0$.

Since $\sum_{j,k}|\lambda_j| \gamma^{-1} 2^{1-k}\leq
4\gamma^{-1}\sum_{j}|\lambda_j| <\infty$, the sum
$f:=\sum_j\lambda_j \sum_k \gamma^{-1} 2^{1-k} \beta_j^k$ defines an
element of $\dot{HS}^1_{2,{\rm ato}}$, with the convergence being in
$\dot{W}_1^1$.  This means that in $\Lone$ we have
$$df=d\left( \sum_{j,k}\lambda_j (b_j\chi_j^k) \right) = \sum_j\lambda_j \sum_k d(b_j\chi_j^k) = g.$$
Therefore $g=df=d\left( \sum_{j,k}\lambda_j (b_j\chi_j^k) \right)$,
with
$\|f\|_{\dot{HS}^1_{2,{\rm ato}}} \leq 4\gamma^{-1}\sum_{j}|\lambda_j|$.
Taking the infimum over all such decompositions of $g$, we see that
$$\inf_{df = g} \|f\|_{\dot{HS}^1_{2,{\rm ato}}}\leq 4\gamma^{-1}\|g\|_{\Honemol(\wedge^1 T^*M)}.$$
\end{proof}

\begin{cor} In the Euclidean case, we then obtain
$$\Honemol(\Rn, \wedge^1)={\mathcal H}^1_d(\Rn, \wedge^1)=d(\dotMone)=d(\dotHStato)$$
for all $t>1$.
(For details on ${\mathcal H}^1_d(\Rn, \wedge^1)$, see \cite{LM}).
\end{cor}


\begin{thebibliography}{100}

\bibitem{ambrosio1}
L.~Ambrosio, M.~Miranda~Jr., D.~Pallara,
\newblock Special functions of bounded variation in doubling metric measure spaces,
\newblock in: Calculus of variations~: topics from the mathematical heritage of
  E. De Giorgi, Quad.\ Mat.\ 14, Dept.\ Math, Seconda Univ. Napoli, Caserta, 2004,
  pp.\ 1--45.


\bibitem{AC}
P.~Auscher,  T.~Coulhon, 
\newblock Riesz transform on manifolds and Poincar\'e inequalities,
 \newblock Ann.\ Sc.\ Nor.\ Sup.\ Pisa 5 (2005), 531--555.
  
\bibitem{AMR}
P.~Auscher, A.~McIntosh, E.~Russ,
\newblock Hardy spaces of differential forms on Riemannian manifolds,
\newblock J.\ Geom.\ Anal.\ 18 (2008), 192--248.

\bibitem{ART}
P.~Auscher, E.~Russ, P.~Tchamitchian,
\newblock Hardy-Sobolev spaces on strongly Lipschitz domains of ${\mathbb{R}}^n$,
\newblock J.\ Funct.\ Anal. 218 (2005), 54--109.

\bibitem{babe}
N.~Badr, F.~Bernicot,
\newblock Abstract Hardy-Sobolev spaces and Interpolation,
\newblock to appear in J.\ Funct.\ Anal.

\bibitem{calderon}
A.\ P.~Calder\'on,
\newblock Estimates for singular integral operators in terms of maximal
functions, in: Collection of articles honoring the completion by
Antoni Zygmund of 50 years of scientific activity, VI,
\newblock Studia Math.\ 44 (1972), 563--582.

\bibitem{cheeger}
J.~Cheeger,
\newblock Differentiability of Lipschitz functions on metric measure spaces,
\newblock Geom.\ Funct.\ Anal.\ 9 (1999), 428--517.

\bibitem{cho}
Y.-K.~Cho, J.~Kim,
\newblock Atomic decomposition on Hardy-Sobolev spaces,
\newblock Studia Math.\ 177(2006),  25--36.

\bibitem{CW1}
R.~Coifman, G.~Weiss,
\newblock Analyse harmonique sur certains espaces homog\`{e}nes,
\newblock Lecture notes in Math., Springer, 1971.

\bibitem{CW2}
R.~Coifman and G.~Weiss,
\newblock Extensions of Hardy spaces and their use in analysis,
\newblock Bull.\ Amer.\ Math.\ Soc.\ 83 (1977), 569--645.


\bibitem{D}
G.~Dafni,
\newblock Hardy Spaces on Strongly Pseudoconvex Domains in $\C^n$
and Domains of Finite Type in $\C^2$,
\newblock Ph.D. Thesis, Princeton University, 1993.

\bibitem{devsha}
R.\ A.~DeVore, R.\ C.~Sharpley,
\newblock Maximal functions measuring smoothness,
\newblock Mem.\ Amer.\ Math.\ Soc.\ 47, 1984.

\bibitem{FHK}
B.~Franchi, P.~Haj{\l}asz, P.~Koskela,
\newblock Definitions of Sobolev classes on metric spaces,
\newblock Annales de l'institut Fourier 49 (1999),
1903--1924.

\bibitem{gatto} A.\ E.~Gatto, C.~Segovia, J.\ R.~Jim\'enez,
\newblock On the solution of the equation $\Delta ^{m}F=f$ for $f\in H^{p}$,
\newblock Conference on harmonic analysis in honor of Antoni Zygmund, 
\newblock Vol. I, II (Chicago, Ill., 1981), 394--415, 
\newblock Wadsworth Math.\ Ser., Wadsworth, Belmont, CA, 1983.

\bibitem{hajlasz1}
P.~Haj{\l}asz,
\newblock Sobolev spaces on an arbitrary metric space,
\newblock Potential Anal.\ 5 (1996), 403--415.

\bibitem{hajlasz2}
P.~Haj{\l}asz,
\newblock Sobolev spaces on metric-measure spaces, in:
Heat kernels and analysis on manifolds, graphs, and metric spaces
(Paris, 2002),
\newblock Contemp.\ Math.\ 338, Amer.\ Math.\ Soc.\ Providence, RI 2003, pp.\ 173--218.

\bibitem{hajkin}
P.~Haj{\l}asz, J.~Kinnunen,
\newblock H\"older quasicontinuity of Sobolev functions on metric spaces,
\newblock Rev.\ Mat.\ Iberoam.\ 14 (1998), 601--622.

\bibitem{hajkosk}
P.~Haj{\l}asz, P.~Koskela,
\newblock Sobolev met Poincar\'{e},
\newblock Mem.\ Amer.\ Math.\ Soc.\ 145 (2000),  1--101.

\bibitem{hei}
J.~Heinonen,
\newblock Lectures on analysis on metric spaces,
\newblock Universitext, Springer-Verlag, New York, 2001.

\bibitem{janson}
S.~Janson,
\newblock On functions with derivatives in $\Hone$,
\newblock in: Harmonic Analysis and Partial Differential Equations (El Escorial, 1987),
\newblock Lecture Notes in Mathematics 1384, Springer, Berlin, 1989, pp.\ 193--201.

\bibitem{kintuo}
J.~Kinnunen, H.~Tuominen,
\newblock Pointwise behaviour of ${M}^{1,1}$ {S}obolev functions,
\newblock Math. Z.\ 257 (2007), 613--630.

\bibitem{KS}
P.~Koskela, E.~Saksman,
\newblock Pointwise characterizations of {H}ardy-{S}obolev functions,
\newblock Math.\ Res.\ Lett.\ 15 (2008), 727--744.

\bibitem{KYZ}
P.~Koskela, D.~Yang, Y.~Zhou,
\newblock A characterization of Haj{\l}asz-{S}obolev and Triebel-Lizorkin spaces via grand Littlewood-Paley functions,
\newblock J.\ Funct.\ Anal. 258 (2010), 2637--2661.

\bibitem{LM}
Z.~Lou, A.~McIntosh,
\newblock Hardy spaces of exacts forms on $\Rn$,
\newblock Trans.\ Amer.\ Math.\ Soc.\  357 (2005),  1469--1496.

\bibitem{LY}
Z.~Lou, S.~Yang,
\newblock An atomic decomposition for the Hardy-Sobolev space,
\newblock Taiwanese J.\ Math.\ 11 (2007),  1167--1176.

\bibitem{maciassegovia}
R.\ A.~Macias, C.~Segovia,
\newblock A decomposition into atoms of distributions on spaces of homogeneous type,
\newblock Advances in Math.\ 33 (1979), 271--309.

\bibitem{MPPP}
M.~Miranda  Jr., D.~Pallara, F.~Paronetto, M.~Preunkert,
\newblock Heat semigroup and functions of bounded variation on Riemannian manifolds,
\newblock J.\ Reine Angew.\ Math.\ 613 (2007), 99--119.

\bibitem{miyachi}
A.~Miyachi,
\newblock Hardy-Sobolev spaces and maximal functions,
\newblock J.\ Math.\ Soc.\ Japan 42 (1990), 73--90.

\bibitem{Stein}
E.\ M.~Stein,
\newblock Harmonic analysis: real-variable methods, orthogonality, and oscillatory integrals,
\newblock Princeton Mathematical Series 43,
\newblock  Princeton University Press, Princeton, NJ, 1993.

\bibitem{strichartz}
R.~Strichartz,
\newblock $H^p$ Sobolev spaces,
\newblock Colloq.\ Math.\ LX/LXI (1990), 129--139.

\bibitem{uchiyama}
A.~Uchiyama,
\newblock A maximal function characterization of $H^p$ on the space of homogeneous type,
\newblock Trans.\ Amer.\ Math.\ Soc.\ 262 (1980),  579--592.

\bibitem{warner}
F. W. Warner,
\newblock Foundations of Differentiable Manifolds and Lie Groups,
\newblock Scott, Foresman and Company, Glenview, Illinois, 1971.

\end{thebibliography}
\end{document}